\documentclass[10pt,a4paper]{article}

\usepackage[utf8]{inputenc}
\usepackage[pdftex]{graphicx}
\usepackage{pdfpages}
\usepackage{amscd,amsmath,amssymb,amsthm,amsfonts,epsfig}
\usepackage{amsmath}
\usepackage{amssymb}
\usepackage{amsthm}
\usepackage{enumerate}
\usepackage{textcomp}
\usepackage{todonotes}
\usepackage[titletoc,toc,title]{appendix}
\usepackage[colorlinks,citecolor={blue},linkcolor={blue},linktocpage=true]{hyperref}
\usepackage[refpage]{nomencl}[0]
\usepackage{cite}  
\usepackage{authblk}
\usepackage[symbol*]{footmisc}

\renewcommand{\pod}[1]{\allowbreak\mathchoice
	{\if@display \mkern 18mu\else \mkern 8mu\fi (#1)}
	{\if@display \mkern 18mu\else \mkern 8mu\fi (#1)}
	{\mkern4mu(#1)}
	{\mkern4mu(#1)}
}
\makeatletter
\renewcommand*{\@fnsymbol}[1]{\@alph{#1})}
\makeatother

\title{On Certain Degenerate Whittaker Models for Cuspidal Representations of $\mathrm{GL}_{k\cdot n}(\bff_q)$}
\date{}

\author{Ofir Gorodetsky\thanks{ofir.goro@gmail.com} \,}
\author{Zahi Hazan\thanks{zahi.hazan@gmail.com}}

\affil{Raymond and Beverly Sackler School of Mathematical Sciences, Tel Aviv University, street, 6997801, Tel Aviv, Israel}

\numberwithin{equation}{section}

\newcommand{\unimat}{U}
\newcommand{\bff}{\mathbb{F}}
\newcommand{\piknpsi}[1]{\pi_{{#1},N,\psi}}

\newcommand{\gl}[2][\bff]{\mathrm{GL}_{#2}(#1)}
\newcommand{\induce}[3][\gl{3n}]{\mathrm{Ind}_{#2} ^{#1} (#3)}

\theoremstyle{plain}

\newtheorem{seclem}{Lemma}[section]
\newtheorem{intTheorem}{Theorem}
\newtheorem{Theorem}[seclem]{Theorem}

\newtheorem{cor}[seclem]{Corollary}
\newtheorem{prop}[seclem]{Proposition}

\theoremstyle{definition}
\newtheorem{Def}[seclem]{Definition}

\theoremstyle{remark}
\newtheorem{Rem}[seclem]{Remark}

\begin{document}
	
\maketitle

\begin{abstract}
Let $\pi$ be an irreducible cuspidal representation of $\gl[\bff_q]{kn}$. Assume that $\pi = \pi_{\theta}$, corresponds to a regular character $\theta$ of $\bff_{q^{kn}}^{*}$. We consider the twisted Jacquet module of $\pi$ with respect to a non-degenerate character of the unipotent radical corresponding to the partition $(n^k)$ of $kn$. We show that, as a $\gl[\bff_q]{n}$-representation, this Jacquet module is isomorphic to $\pi_{\theta \upharpoonright_{\bff_n^*}} \otimes \mathrm{St}^{k-1}$, where $\mathrm{St}$ is the Steinberg representation of $\gl[\bff_q]{n}$. This generalizes a theorem of D. Prasad, who considered the case $k=2$.
We prove and rely heavily on a formidable identity involving $q$-hypergeometric series and linear algebra.
\end{abstract}

\section{Introduction} \label{introduction} \label{sec:intro}
Let $\bff:=\bff_q$ be a finite field of size $q$.  We will fix a nontrivial character $\psi_0$ of $\bff$. Denote by $\bff_m:=\bff_{q^m}$ the unique degree $m$ field extension of $\bff$.  Let $k$ be a positive integer. Denote the diagonal subgroup of $\left(\gl{\ell}\right)^r$ by
\begin{equation}
	\Delta^r\left(\gl{\ell}\right):=\left\{\left. \left(g,\ldots,g\right)\in \left(\gl{\ell}\right)^r\ \right| g\in \gl{\ell} \right\}.
\end{equation}
For a partition $\rho=\left(k_1,k_2,\ldots,k_s\right)$ of $\ell$, consider the corresponding standard parabolic subgroup $P_{\rho}$ of $\gl{\ell}$, and $M_{\rho}$, $N_{\rho}$ be the corresponding Levi part and unipotent radical.

We begin by describing a theorem of Prasad \cite[Thm.~1]{prasad2000}. Here we consider $\gl{2n}$, and denote $P^\prime=P_{n,n}$, $M^\prime = M_{n,n}$ and $N^\prime =N_{n,n}$. Consider the following character $\psi^\prime$ of $N^\prime$,
\begin{equation*}
\psi^\prime \left( \begin{pmatrix}
I_n & X \\ 0 & I_n
\end{pmatrix} \right) :=\psi_{0}\left(\mathrm{tr}\left(X\right)\right).
\end{equation*}
Let $\pi^\prime$ be an irreducible representation of $\gl{2n}$ acting on a space $V_{\pi^\prime}$.
Define
\begin{equation*}
{V_{\pi^\prime_{N^\prime,\psi^\prime}}}:=\left\{ v\in V_{\pi^\prime}\ \left|\ \pi^\prime(u)v=\psi^\prime(u)v,\ \forall u\in N^\prime\right. \right\}.
\end{equation*}
This is the $\left(N^\prime,\psi^\prime \right) $-isotypic subspace of $V_{\pi^\prime}$. We know that $V_{\pi^\prime_{N^\prime,\psi^\prime}}$ is the image of the canonical projection of $V_{\pi^\prime}$ on $V_{\pi^\prime_{N^\prime,\psi^\prime}}$ given by
\begin{equation*}
P_{N^\prime,\psi^\prime}\left(v\right)=\frac{1}{\left|N^\prime\right|}\sum_{u\in N^\prime} \overline{\psi^\prime}\left(u\right)\pi^\prime(u)v.
\end{equation*}
Since $\mathrm{tr}\left(gXg^{-1}\right)=\mathrm{tr}\left(X\right)$ for all $g\in \gl{n}$, and by identifying  $\Delta^2\left(\gl{n}\right)$ with $\gl{n}$, it follows that $V_{\pi^\prime_{N^\prime,\psi^\prime}}$ is a representation space for $\gl{n}$. The space $V_{\pi^\prime_{N^\prime,\psi^\prime}}$
is referred to as the twisted Jacquet module of the space $V_{\pi^\prime}$ with respect to $\left(N^\prime,\psi^\prime\right)$. Prasad proved the following theorem.
\begin{Theorem}\cite[Thm.~1]{prasad2000}\label{thmprasad}
Let $\pi^\prime$ be an irreducible cuspidal representation of $\gl{2n}$ obtained from a character $\theta$ of $\bff^*_{2n}$. Then
\begin{equation}
   \pi^\prime_{N^\prime,\psi^\prime} \cong \mathrm{Ind}_{\bff^*_n}^{\gl{n}}\theta\upharpoonright_{\bff_n^*}.
\end{equation}
\end{Theorem}
Prasad proved this theorem by an explicit calculation of the characters of  $\pi^\prime_{N^\prime,\psi^\prime}$ and of the induced representation $\mathrm{Ind}_{\bff^*_n}^{\gl{n}}\theta\upharpoonright_{\bff_n^*}$. At any element of $\gl{n}$ the characters are the same. Therefore, the two representations are equivalent.

\vspace{9pt}
Fix $k \ge 1$. Let $\rho=(n^k)$ be the partition of $kn$ consisting of $k$ parts of size $n$. In this paper we denote $G:=\gl{kn}$, $P=P_{\rho}$,  $M=M_{\rho}$ and $N=N_{\rho}$.
We have the Levi decomposition $P=M\ltimes N$. We write $\unimat \in N$ in the form
\begin{equation}\label{udef}
	\quad \unimat= \begin{pmatrix}
		I_{n} & X_{1,1} & X_{1,2} & \cdots&X_{1,k-2}&X_{1,k-1} \\
		0 & I_{n} & X_{2,2} & \cdots &X_{2,k-2}&X_{2,k-1} \\
		0& 0 & I_{n} & \cdots &X_{3,k-2}& X_{3,k-1}\\
		\vdots&\vdots&\vdots&&\vdots&\vdots\\
		0 & 0 & 0& \cdots & I_n &X_{k-1,k-1}\\
		0 & 0 & 0& \cdots & 0 &I_n
	\end{pmatrix},
\end{equation}
where the matrices $X_{i,j}$ ($1\leq i\leq j \leq k-1$) are elements of $M_n(\bff)$.

\begin{Def}
A character $\psi:N\to\mathbb{C}^*$ is said to be non-degenerate if it has the form
$$\psi \left( \unimat \right) :=\psi_{0}\left(\mathrm{tr}\left(\sum_{i=1}^{k-1} A_i X_{i,i}\right)\right)=\prod_{i=1}^{k-1}\psi_{0}\left(\mathrm{tr}\left(A_i X_{i,i}\right)\right),$$
where the matrices $A_i$ are invertible.
\end{Def}
Let $\psi:N\to\mathbb{C}^{*}$ be a non-degenerate character. Let $\pi$ be an irreducible representation of $G$, acting on a space $V_\pi$. We will denote by $V_{\pi_{k,N,\psi}}$ the largest subspace of $V_\pi$, on which $N$ operates through $\psi$, i.e.
\begin{equation}
{V_{\pi_{k,N,\psi}}}=\left\{ v\in V_{\pi}\ \left|\ \pi(\unimat)v=\psi(\unimat)v,\ \forall \unimat \in N\right. \right\}.
\end{equation}
This is the $\left(N,\psi \right) $-isotypic subspace of $V_{\pi}$. Same as before, $V_{\pi_{k,N,\psi}}$ is the image of the canonical projection of $V_\pi$ on $V_{\pi_{k,N,\psi}}$ given by
\begin{equation} \label{projform}
P_{k,N,\psi}\left(v\right)=\frac{1}{\left|N\right|}\sum_{\unimat\in N}\overline{\psi}\left(\unimat \right)\pi(\unimat)v.
\end{equation}
Since $M$ normalizes $N$, it acts on the characters of $N$ as follows. If $m \in M$, then for all $u \in N$
\begin{equation*}
(m\cdot \psi)(\unimat)=\psi \left(m^{-1}\unimat m\right).
\end{equation*}
We have, for $m\in M$,
\begin{equation*}
\pi(m)V_{\pi_{k,N,\psi}}=V_{\pi_{k,N,m\cdot\psi}}.
\end{equation*}
Let us compute the stabilizer of $\psi$ in $M$. We write
\begin{equation*}
m=\begin{pmatrix}
B_{1}&0&\cdots & 0\\
0&B_{2}&\cdots & 0\\
\vdots
&\vdots &&\vdots\\
0&0& \cdots& B_{k}
\end{pmatrix},\quad \forall 1\leq i\leq k: B_{i}\in \gl{n},
\end{equation*}
Then
\begin{equation*}
(m\cdot \psi)(\unimat)=\psi_0 \left(\text{tr}\left( \sum_{i=1}^{k-1}A_i B_i^{-1}  X_{i,i}B_{i+1}\right)\right).
\end{equation*}
Thus, $m\cdot \psi=\psi$ iff $B_i=B_{i+1}$ for all $1 \le i \le k-1$. In other words,
\begin{equation*}
\text{stab}_M\psi = \Delta^k\left(\gl{n}\right) \cong \gl{n}.
\end{equation*}
 Therefore, $V_{\pi_{k,N,\psi}}$ is a $\gl{n}$-module. We denote by $\pi_{k,N,\psi}$ the resulting representation of $\gl{n}$ in $V_{\pi_{k,N,\psi}}$. It is easy to see that by conjugation with an element in the Levi part, we may simply take all the $A_i$ to be the identity matrix. The corresponding twisted Jacquet modules are isomorphic. In the rest of the paper we assume $A_i=I_n$ and fix
 $\psi \left( \unimat \right) :=\psi_{0}\left(\mathrm{tr}\left(\sum_{i=1}^{k-1} X_{i,i}\right)\right)$.

 The goal of this paper is to calculate the character of $\pi_{k,N,\psi}$, and to describe it by known representations, for an irreducible, cuspidal representation $\pi=\pi_\theta$ of $\gl{kn}$, associated to a regular character $\theta$ of $\bff_{kn}^*$. The paper generalizes Prasad's result for the case $k=2$.
 The methods used in this paper are generalizations of the methods used by the second author in his thesis \cite{hazan2016} for the case $k=3$. From the character calculation, done in Theorem \ref{int_char_thm}, we were able to describe in Theorem \ref{mainthrm} $\pi_{k,N,\psi}$ by the representations $\mathrm{Ind}_{\bff^*_{\ell}} ^{\gl{n}} \theta \upharpoonright _{\bff^*_{\ell}}$, where $\ell \mid n$. Furthermore, we give a compact description of $\pi_{k,N,\psi}$ by the Steinberg representation in the following theorem.
\begin{intTheorem} \label{st_thm}
	Let $k\ge 1$. Let $\pi_\theta$ be an irreducible cuspidal representation of $\gl{kn}$ obtained from a character $\theta$ of $\bff^*_{kn}$.
	Then
	\begin{equation} \label{st_cong}
	\piknpsi{k}\cong \pi_{\theta\upharpoonright_{\bff^*_n}}\otimes \mathrm{St}^{k-1},
	\end{equation}
	where $\pi_{\theta\upharpoonright_{\bff^*_n}}$ is the irreducible cuspidal representation of $\gl{n}$ obtained from $\theta\upharpoonright_{\bff^*_n}$ and $\mathrm{St}$ is the Steinberg representation of $\gl{n}$.
\end{intTheorem}
Note that for $n=1$, Theorem \ref{st_thm} gives $\pi_{k,N ,\psi } \cong \theta\upharpoonright_{\bff^*}$, which also follows from Gel'fand-Graev \cite{gelfand1962} in case of $\gl{k}$ (cf. \cite[Ch.~8.1]{carter1993}).

We are currently investigating the analogue construction for $\bff$ a non-Archimedean local field.
\subsection{Structure of paper}
 In \S \ref{dim_sec} we calculate the dimension of $\pi_{k,N,\psi}$. Green's formula allows us to express the dimension as rather complicated sum. We used tools from $q$-hypergeometric series and linear algebra to show that this sum admits the following compact form.
 \begin{intTheorem} \label{int_dim}
Let $k\ge 2$. We have
\begin{equation*}
 	\mathrm{dim}\left(\pi_{k,N,\psi}\right)=q^{(k-2) \frac{n(n-1)}{2}}\frac{\left|\gl{n}\right|}{q^n-1}.
 	\end{equation*}
 \end{intTheorem}
We denote the character of $\pi_{k,N,\psi}$ by $\Theta _{k,N ,\psi}$. In \S \ref{char_sec} we compute $\Theta _{k,N ,\psi}$, which apart from the tools used in Theorem \ref{int_dim} requires understanding of some conjugacy classes of $\gl{n}$.
\begin{intTheorem} \label{int_char_thm}\
	Let $k\ge 2$. Let $g = s \cdot u$ be the Jordan decomposition of an element $g$ in $\gl{n}$.
	\begin{itemize}
		\item[(I)]  If the semisimple part $s$ does not come from $\bff_n$, then
		\begin{equation*}
		\Theta _{k,N ,\psi }(g)  =0.
		\end{equation*}
		\item[(II)]  If $u\neq I_n$, then
		\begin{equation*}
		\Theta _{k,N ,\psi }(g)  =0.
		\end{equation*}
		\item[(III)] Assume that $u=I_n$ and that the semisimple element $s$ comes from $\bff_d\subseteq \bff_n$ and $d\mid n$ is minimal. Let $\lambda$ be an eigenvalue of $s$ which generates $\bff_d$ over $\bff$. Then,
		\begin{align*}
			\Theta _{k,N ,\psi }(s)  = (-1)^{k(n-{d^\prime})}  q^{(k-2)\frac{n({d^\prime}-1)}{2}} \cdot \left[ \sum \limits _{i=0}^{d-1} \theta(\lambda^{q^i})\right] \cdot  \frac{\left|\mathrm{GL}_{{d^\prime}}(\bff_{d})\right|}{q^n-1} ,
		\end{align*}
		where $d^\prime=n/d$.
	\end{itemize}
\end{intTheorem}
For any $\ell$ dividing $n$ and any $k\ge 2$, let
\begin{equation}\label{anlq}
a_{k;n,\ell}(q)=\frac{q^\ell-1}{q^n-1} \sum_{m: \, \ell \mid m \mid n} \mu\left(\frac{m}{\ell}\right) (-1)^{k(n- \frac{n}{m})} q^{(k-2)\frac{n}{2} (\frac{n}{m}-1)},
\end{equation}
where $\mu$ is the M\"obius function. When $k=2$, it is easily shown (see Lemma \ref{lemand}) that
\begin{equation}
a_{2;n,\ell}(q) = \delta_{\ell,n}.
\end{equation}
If $k>2$ we show in Lemma \ref{lemand} that
$a_{k;n,\ell}(q)\in\mathbb{N}_{>0}$, except when $k$ is odd, $n$ is even and $2 \nmid \frac{n}{\ell}$, in which case $-a_{k;n,\ell}(q)\in\mathbb{N}_{>0}$. In
\S \ref{mainth_sec} we conclude from Theorem \ref{int_char_thm} and Lemma \ref{lemand} the following decomposition of representations.

\begin{intTheorem} \label{mainthrm}
Let $k \ge 2$.
	\begin{itemize}
		\item[(I)] If $k$ is even or $n$ is odd, we have
		\begin{equation}
			\begin{split}
				\pi _{k,N,\psi} \cong &\bigoplus_{\ell \mid n} a_{k;n,\ell}(q) \cdot \mathrm{Ind}_{\bff^*_{\ell}} ^{\gl{n}} \theta \upharpoonright _{\bff^*_{\ell}}. \label{mainth1}
			\end{split}
		\end{equation}
		\item[(II)] If $k$ is odd and $n$ is even, we have
		\begin{equation}
			\pi _{k,N,\psi}\oplus  \bigoplus_{\ell: \, \ell \mid n, 2 \nmid \frac{n}{\ell}}(-a_{k;n,\ell}(q))\cdot \mathrm{Ind}_{\bff^*_{\ell}} ^{\gl{n}} \theta \upharpoonright _{\bff^*_{\ell}} \cong  \bigoplus_{\ell: \, \ell \mid n, 2 \mid \frac{n}{\ell}} a_{k;n,\ell}(q) \cdot \mathrm{Ind}_{\bff^*_{\ell}} ^{\gl{n}} \theta \upharpoonright _{\bff^*_{\ell}}. \label{mainth2}
		\end{equation}
		
	\end{itemize}
	
\end{intTheorem}
 In \S \ref{st_sec} we deduce Theorem \ref{st_thm} from Theorem \ref{int_char_thm}.

\section{Preliminaries}
\subsection{Cuspidal representations}
We review the irreducible cuspidal representations of $\gl{m}$ as in S.I. Gel'fand \cite[\S6]{gelfand1975} (originally in J. A. Green \cite{green1955}). Irreducible cuspidal representations of $\gl{m}$, from which all the other irreducible representations of $\gl{m}$ are obtained via the process of parabolic induction, are associated to regular characters of $\bff_{m}^*$. A multiplicative character $\theta$ of $\bff_{m}^*$ is called \textit{regular} if, under the action of the Galois group of $\bff_{m}$ over $\bff$, the orbit of $\theta$ consists of $m$ distinct characters of $\bff_{m}^*$.

We denote the irreducible cuspidal representation of $\gl{m}$ associated to a regular character $\theta$ of $\bff_{m}^*$ by $\pi_\theta$  and the character of the representation $\pi _\theta$ by $\Theta_\theta$.

Given $a\in\bff_{m}$, consider the map $m_a:\bff_{m}\to \bff_{m}$, defined by $m_a(x)=ax$. The map $a\mapsto m_a$ is an injective homomorphism of algebras $\bff_{m} \hookrightarrow \mathrm{End}_{\bff}(\bff_{m})$. This way, every element of $\bff_{m}^*$ gives rise to a well-defined conjugacy class in $\gl{m}$. The elements in the conjugacy classes in $\gl{m}$, which are so obtained from elements of $\bff_{m}^*$, are said to come from $\bff_{m}^*$.

We summarize the information about the character $\Theta_\theta$ in the following theorem. We refer to the paper of S. I. Gel'fand \cite[\S6]{gelfand1975} for the statement of this theorem in this explicit form, which is originally due to Green \cite[Thm.~14]{green1955} (See also the paper of Springer and Zelevinsky \cite{springer1984})  The theorem is quoted as it appears in \cite[Thm.~2]{prasad2000}.
\begin{Theorem}[Green \cite{green1955}] \label{greenthm}
	Let $\Theta_\theta$ be the character of a cuspidal representation $\pi_\theta$ of $\gl{m}$ associated to a regular character $\theta$ of $\bff_{m}^*$. Let $g = s \cdot u$ be the Jordan decomposition of an element $g$ in $\gl{m}$ ($s$ is a semisimple element, $u$ is unipotent and $s,u$ commute). If $\Theta_\theta(g) \not= 0$, then the semisimple element $s$ must come from $\bff_{m}^*$. Suppose that $s$ comes from $\bff_{m}^*$. Let $\lambda$ be an eigenvalue of $s$ in $\bff_{m}^*$, and let $t=\mathrm{dim}_{\bff_{m}}\ker (g-\lambda I)$. Then
	\begin{equation}
	\Theta_\theta(s\cdot u) = (-1)^{m-1}\left[\sum\limits_{\alpha=0}^{d-1}\theta(\lambda^{q^\alpha})\right](1-q^d)(1-({q^d})^2)\cdots(1-({q^d})^{t-1}) \label{green}
	\end{equation}
	where $q^d$ is the cardinality of the field generated by $\lambda$ over $\bff$, and the summation is over the various distinct Galois conjugates of $\lambda$.
\end{Theorem}
\begin{cor}\label{greencor}
	The value $\Theta_\theta(g)$ is determined by the eigenvalue of $g$ and the number of Jordan blocks of $g$, which, in turn, is determined by  $\mathrm{dim}_{\bff_{m}}\ker (g-\lambda I)$.
\end{cor}

\subsection{Characters induced from subfields} \label{char_ind_sec}

The following lemma summarizes the information about the character of $\induce[\gl{n}]{\bff_\ell ^*}{\theta\upharpoonright_{\bff_\ell^*}}$, where $\ell\mid n$ and $\theta$ is a character of $\bff_n^*$.
\begin{seclem} \label{lemindch}
Let $\theta$ be a character  of $\bff^*_n$. Suppose that $s\in \gl{n}$ comes from $\bff_d\subseteq \bff_\ell$ ($d\mid \ell$ is minimal). Then, the character $\Theta_{\mathrm{Ind}_\ell}$ of $\induce[\gl{n}]{\bff_\ell ^*}{\theta\upharpoonright_{\bff_\ell^*}}$ at $s$ is given by
\begin{align}
\Theta_{\mathrm{Ind}_\ell}(s)={}&\frac{1}{q^\ell-1}\sum _{\substack{g\in \gl{n}\\g^{-1}sg\in \bff^*_\ell}}\theta(g^{-1}sg) \label{indchar1}\\
={}&\frac{\left|\gl[\bff_d]{d^\prime}\right|}{q^\ell-1}\left[\sum\limits_{i=0}^{d-1}\theta(\lambda^{q^i})\right], \label{indchar2}
\end{align}
where $d^\prime = n/d$. The last sum is over the different Galois conjugates of $s$, thought of as an element of $\bff_d$. The value of the character $\Theta_{\mathrm{Ind}_\ell}$ at an element of $\gl{n}$ which does not come from $\bff_\ell$ is zero.
\end{seclem}

\begin{Rem} \label{rem_conj_in_fn}
	Recall that in \eqref{indchar1} $\bff_\ell^*$ is considered a subgroup of $\gl{n}$ by the injective map $a\mapsto [m_a]$, where $[m_a]$ is the representing matrix of $m_a$ with	 respect to a fixed basis of $\bff_{n}$ over $\bff$.
	Note that the choice of basis for $[m_a]$ does not affect the values of $\Theta_{\mathrm{Ind}_\ell}$.
\end{Rem}

\subsection{On some conjugacy classes of $\gl{n}$} \label{conj_class}
\subsubsection{Analogue of Jordan form} \label{anal_jord_sec}
Let $g\in \gl{n}$ and $g=s\cdot u$ its Jordan decomposition. Assume that $s$ comes from $\bff_d\subseteq \bff_n$ ($d\mid n$ is minimal). Let $\lambda\in \bff_d^*$ be an eigenvalue of $s$, which generates the field $\bff_d$ over $\bff$. Denote by $f$ the characteristic polynomial of $\lambda$ (of degree $d$), and by $L_f\in \gl{d}$ the companion matrix of $f$. For $\ell\geq 1$ we denote
\begin{equation}
	L_{f,\ell}=\begin{pmatrix}
	L_f&I_d&&\\&L_f&&\\&&\ddots&I_d\\&&&L_f
	\end{pmatrix}\in\gl{\ell\cdot d}.
\end{equation}
This is an analogue of a Jordan block. As in \cite{gelfand1975,green1955}, there exists $\rho=\left(\ell_1,\ldots,\ell_r\right)$, a partition of $\frac{n}{d}$, $\ell_1\geq\ell_2\geq\ldots\geq\ell_r$, $\frac{n}{d}=\sum_{i=1}^r\ell_i$, such that $g$ is conjugate to
\begin{equation}
L_{\rho}(f):=\begin{pmatrix}
L_{f,\ell_1}&&&\\&L_{f,\ell_2}&&\\&&\ddots&\\&&&L_{f,\ell_r}
\end{pmatrix},
\end{equation}
i.e. there exists $R\in\gl{n}$ such that
\begin{equation}\label{rgdef}
R^{-1}gR=L_\rho(f).
\end{equation}
Notice that in case $u=I_n$ ($g$ is semisimple), we have $\rho=(1^{n/d})$ and there exists $R\in\gl{n}$ such that $R^{-1}gR$ is a diagonal block matrix with $d^\prime=n/d$ times $L_f$ on the diagonal.
Otherwise, $\ell_1>1$ and, in particular, there exists $R\in\gl{n}$ such that the upper $2d\times 2d$ left corner of  $R^{-1}gR$ is
\begin{equation}
\begin{pmatrix}
L_f&I_d\\ &L_f
\end{pmatrix}.
\end{equation}
Now, $s$ (and so $g$) has $d$ different eigenvalues obtained by applying the Frobenius automorphism $\sigma$, which generates the Galois group $\mathrm{Gal}(\bff_d/\bff)$, namely
\begin{equation*}
\left\{ \lambda , \sigma (\lambda), \ldots , \sigma ^{d-1} (\lambda) \right\}=\left\{ \lambda , \lambda ^q, \ldots , \lambda ^{q^{d-1}} \right\},
\end{equation*}
all of multiplicity ${d^\prime}=n/d$ in the characteristic polynomial of $s$.
Let $0\neq v_0\in\bff_d^d$ satisfy $L_f\cdot v_0=\lambda v_0$. So $L_f\cdot \sigma^{i}(v_0)=\lambda^{q^i}\sigma^{i}(v_0)$, for $0\leq i\leq d-1$. Hence, $B=\left\{
v_0,\sigma(v_0),\ldots,\sigma^{d-1}(v_0)\right\}\subseteq\bff_d^d$ is linearly independent over $\bff_d$, since its elements are eigenvectors of $L_f$ for different eigenvalues. Let $T\in\gl[\bff_d]{d}$ be the diagonalizing matrix of $L_f$ obtained by $B$, i.e.
\begin{equation}\label{tgdef}
T^{-1}L_fT=D,
\end{equation}
where
\begin{equation}
D:=\mathrm{diag}\left(\lambda,\ldots,\lambda^{q^{d-1}}\right).
\end{equation}
Denote by $\Delta^{d^\prime}\left(T\right)$ the diagonal block matrix with $d^\prime$ times $T$ on the diagonal. Explicitly, the columns of $\Delta^{d^\prime}\left(T\right)$ are the vectors of the basis
\begin{equation} \label{eigenvectors}
C=\left\{v_0(i,j)\right\}_{0\leq i\leq d-1}^{0\leq j\leq d^\prime-1},
\end{equation}
whose $(j\cdot d+i)$-th vector is given by
\begin{equation}
	v_0(i,j)=\begin{pmatrix}
		\underline{0}_{j\cdot d}\\\sigma^{i}(v_0)\\\underline{0}_{n-(j+1)\cdot d}
	\end{pmatrix}\in \bff_d^n,
\end{equation}
where ${0\leq i\leq d-1}$ and ${0\leq j\leq d^\prime-1}$.
Thus, in case $u=I_n$
\begin{equation}
\Delta^{d^\prime}\left(T^{-1}\right)R^{-1}gR\Delta^{d^\prime}\left(T\right)=\begin{pmatrix}
D& & \\ &\ddots & \\ & &D
\end{pmatrix}.
\end{equation}
Otherwise
\begin{equation}
\Delta^{d^\prime}\left(T^{-1}\right)R^{-1}gR\Delta^{d^\prime}\left(T\right)=\begin{pmatrix}
D&I_d& & & \\ &D& & & \\ & &D&*&\\ & & &\ddots &\\	& & & &D
\end{pmatrix}.
\end{equation}
We denote
\begin{equation} \label{jord_analog}
g_{\rho}:=g_{\rho,R}=\Delta^{d^\prime}\left(T^{-1}\right)R^{-1}gR\Delta^{d^\prime}\left(T\right).
\end{equation}
The matrix $g_{\rho}$ is sometimes referred to as an analogue of the Jordan form of $g$ \cite[\S0]{gelfand1975}.
\subsubsection{Conjugating an arbitrary matrix} \label{cong_arb_mat}
Let $A\in M_n(\bff)$. We use the notation of \S \ref{anal_jord_sec}. In particular, we have a fixed $g \in \gl{n}$ and corresponding $R$ and $T$ as defined in \eqref{rgdef} and \eqref{tgdef}. We will study the following conjugation
\begin{equation}
A_{\rho}:=A_{\rho,R}=\Delta^{d^\prime}\left(T^{-1}\right)R^{-1}AR\Delta^{d^\prime}\left(T\right) \in M_n(\bff_d).
\end{equation}
Since $R\in\gl{n}$, $A\mapsto R^{-1}AR$ is an isomorphism. Hence, there exists a unique $A_R$ such that $A=RA_RR^{-1}$ so we write $A_\rho=\Delta^{d^\prime}\left(T^{-1}\right)A_R\Delta^{d^\prime}\left(T\right)$.

Let $B \in M_n(\bff_d)$. Let us represent the vectors $B\cdot v_0(0,m)$, for any $0\leq m \leq {d^\prime}-1$, as a linear combination of the basis $C$ given in \eqref{eigenvectors}:
\begin{equation} \label{lincombA}
B\cdot v_0(0,m) = \sum \limits _{\substack{0\leq i \leq d-1\\ 0\leq j \leq d^\prime-1}} a_{m,i;j} \cdot v_0(i,j),\qquad  a_{m,i;j}\in\bff_d.
\end{equation}
A necessary and sufficient condition for $B \in M_n(\bff)$ is that for all $0\leq m \leq {d^\prime}-1,\ 0\leq r \leq d-1$,
\begin{equation} \label{conditionsA}
B\cdot {v_0\left(r,m\right)} = \sum \limits _{\substack{0\leq i \leq d-1\\ 0\leq j \leq d^\prime-1}} \sigma ^r( a_{m,i;j})\cdot v_0\left(i+r\pmod d,j\right).
\end{equation}
By taking $B = A_R \in M_n(\bff)$, we get that \eqref{conditionsA} holds for $A_R$. Therefore, $[A_R]_C = A_{\rho}$ is a $d^\prime \times d^\prime$ matrix with entries from $M_d\left(\bff_d\right)$. For $0\leq m,j\leq d^\prime-1$, the $m$-th row and $j$-th column of $A_{\rho}$, denoted by $A_{m,j}$, is given by
\begin{equation} \label{asig_blocks}
A_{m,j}=\left(\sigma^r\left(a_{m,i-r\pmod d;j}\right)\right)_{0\leq i,r \leq d-1}\ \ ,
\end{equation}
i.e. $A_{m,j}\in M_d\left(\bff_d\right)$ and for ${0\leq i,r \leq d-1}$, the $i$-th row and $r$-th column of $A_{m,j}$ is $\sigma^r\left(a_{m,i-r\pmod d;j}\right)$. We proved,
\begin{seclem} \label{iso_f_to_fd}
	In the above notations, the map $A\mapsto A_\rho$ induces an isomorphism $M_n(\bff)\to  M_{n\times d^\prime}(\bff_d) \cong \left[M_{d\times d^\prime}(\bff_d)\right]^{d^\prime}$, viewed as $\bff$-vector space. It is given by
	\begin{equation}
	  A\mapsto
	  \begin{pmatrix}
		\left(a_{0,i;j}\right)_{\substack{0\leq i \leq d-1\\ 0\leq j \leq d^\prime-1}}\\
		\vdots\\
		\left(a_{d^\prime-1,i;j}\right)_{\substack{0\leq i \leq d-1\\ 0\leq j \leq d^\prime-1}}
	 \end{pmatrix},
	 \end{equation}
	  where the $(m\cdot d+i)$-th row and $j$-th column of the image of $A$ is $a_{m,i;j}\in\bff_d$, for $0\leq m,j\leq d^\prime-1$ and ${0\leq i \leq d-1}$.
\end{seclem}

\subsubsection{Trace under conjugation}
For $g\in\gl{n}$ and $A\in M_n(\bff)$ we will be interested in $\mathrm{tr}\left(g^{-1}A\right)$.  We use the notation of \S \ref{anal_jord_sec} and \S \ref{cong_arb_mat}.
By \eqref{jord_analog}, we have
\begin{equation*}
\mathrm{tr}\left(g^{-1}A\right) =\mathrm{tr}\left(g_{\rho}^{-1}A_\rho\right) .
\end{equation*}
The inverse of an analogue of a Jordan block of order $d\cdot\ell$, is given by
\begin{equation} \label{inverse_jord}
	\left(\begin{pmatrix}
		D&I_d&\\&\ddots&I_d\\&&D
	\end{pmatrix}^{-1}\right)_{i,j}=\begin{cases}
		(-1)^{j-i}D^{-j+i-1},&i\leq j\\0,&i>j,
	\end{cases}
\end{equation}
for $0\leq i,j\leq \ell$, where the LHS of \eqref{inverse_jord} denotes the block in the $i$-th row and $j$-th column. Each block is in $M_d\left(\bff_d\right)$.
Therefore, in case $g$ is not semisimple, we have that $g_{\rho}^{-1}$ is an upper triangular block matrix, with $D^{-1}$ appearing $d^\prime$ times on the diagonal and at least one signed negative power of $D$ appearing in a block above the diagonal. Hence,
\begin{equation} \label{gen_trace}
\begin{split}
\mathrm{tr}\left(g_{\rho}^{-1}A_{\rho}\right)&= \sum_{m=0}^{d^\prime-1}\mathrm{tr}\left(D^{-1}A_{m,m}+D^{-2}\alpha_m\left(g,D^{-1},A_\rho\right)\right) \\
&=\mathrm{tr}\left( \sum_{m=0}^{d^\prime-1} D^{-1}A_{m,m} \right) + \sum_{m=0}^{d^\prime-1} \mathrm{tr}\left(  D^{-2}\alpha_m\left(g,D^{-1},A_\rho\right)\right),
\end{split}
\end{equation}
where $\alpha_m\left(g,D^{-1},A_\rho\right)$, for $0\leq m \leq d^\prime-1$. Notice, that in case $g$ is semisimple, then $\alpha_m\left(g,D^{-1},A_\rho\right)=0$ for all $0\leq m \leq d^\prime-1$. Otherwise, for $0\leq m \leq d^\prime-1$, $D^{-2}\alpha_m\left(g,D^{-1},A_\rho\right)$ equals to a sum of terms of the form $(-1)^{\ell}D^{-\ell-1}A_{\ell,m}$, where $m<\ell \leq d^\prime-1$.

By \eqref{asig_blocks} we have
\begin{equation}
	D^{-1}A_{m,m}=\left(\left(\lambda^{-1}\right)^{q^r}\sigma^r\left(a_{m,{i-r\pmod d};m}\right)\right)_{1\leq i,r \leq d-1}.
\end{equation}
So the first sum in the RHS of \eqref{gen_trace} becomes
\begin{equation*}
\sum_{m=0}^{d^\prime-1}\sum_{r=0}^{d-1}\left(\lambda^{-1}\right)^{q^r}\sigma^r\left(a_{m,0;m}\right)=\sum_{r=0}^{d-1}\sigma^r\left(\lambda^{-1}\sum_{m=0}^{d^\prime-1}a_{m,0;m}\right)=
\mathrm{Tr}_{\bff_{d}/\bff}\left(\lambda ^{-1} \cdot \sum \limits _{m=0} ^{d^\prime-1} a_{m,0;m}\right).
\end{equation*}
On the other hand, for each $0\leq m\leq d^\prime -1$, the term $\mathrm{tr}\left(D^{-2}\alpha_m\left(g,D^{-1},A_\rho\right)\right)$ in \eqref{gen_trace} does not depend on the elements $a_{\ell,0;m}$, where $\ell=m$ (only on $\lambda$ and on $a_{\ell,i,m}$ where $\ell>m$).
Thus, we proved
\begin{seclem} \label{lem_tr_conj}
	In the above notations,
\begin{equation}
\mathrm{tr}\left(g^{-1}A\right)=
\mathrm{Tr}_{\bff_{d}/\bff}\left(\lambda ^{-1} \cdot \sum \limits _{m=0} ^{d^\prime-1} a_{m,0;m}\right)+\sum_{m=0}^{d^\prime-1}\mathrm{tr}\left(D^{-2}\alpha_m\left(g,D^{-1},A_\rho\right)\right),
\end{equation}
and each summand $\mathrm{tr}\left(D^{-2}\alpha_m\left(g,D^{-1},A_\rho\right)\right)$ is independent of $a_{m,0;m}$ appearing in the first summand, for all $0\leq m \leq d^\prime-1$.

In case $g=s$ is semisimple we have
\begin{equation}
\mathrm{tr}\left(g^{-1}A\right)=
\mathrm{Tr}_{\bff_{d}/\bff}\left(\lambda ^{-1} \cdot \sum \limits _{m=0} ^{d^\prime-1} a_{m,0;m}\right).
\end{equation}
\end{seclem}

\subsection{$q$-hypergeometric identity}
In order to calculate the dimension of $\pi_{k,N,\psi}$, we need a combinatorial identity related to ranks of triangular block matrices. Before we present the identity, we prove a lemma that will be needed in its proof. The lemma is a special case of a $q$-analogue of the Chu-Vandermonde identity, phrased in a manner that will be useful for us.
We recall the definition of the $q$-Pochhammer symbol:
\begin{equation}
(a;q)_n = \prod_{i=0}^{n-1}(1-aq^i).
\end{equation}
\begin{seclem}\label{rnmrchu}
Let $R_q(n,m,r)$ be the number of $n \times m$ matrices of rank $r$ over a finite field of size $q$ ($n$, $m$ may be $0$, by the convention that the empty matrix has rank $0$). Let $a$ be an integer greater or equal to $n+m$. Then
\begin{equation}\label{equichu}
\sum_{r \ge 0} R_q(n,m,r) (q;q)_{a-r} = q^{nm}\frac{(q;q)_{a-n}(q;q)_{a-m}}{(q;q)_{a-n-m}}.
\end{equation}
\end{seclem}
\begin{proof}
We start by stating a $q$-analogue of the Chu-Vandermonde identity \cite[Eq.~(1.5.2)]{gasper2004}:
\begin{equation}\label{gaspchu}
\sum_{r \ge 0} \frac{  (q^{-i};q)_r   (b;q)_r  }{(c;q)_r (q;q)_r } \left( \frac{cq^i}{b} \right)^r= \frac{(c/b;q)_i}{(c;q)_i},
\end{equation}
where $i$ is a non-negative integer. Note that the LHS of \eqref{gaspchu} is terminating, since the terms corresponding to $r>i$ vanish. The identity \eqref{gaspchu} is valid whenever $b \neq 0$ and $c \notin \{ q^{-1},\cdots,q^{-(i-1)} \}$. Choosing $i=n$, $b=q^{-m}$, $c=q^{-a}$, we obtain
\begin{equation}\label{chuapp}
\sum_{r \ge 0} \frac{  (q^{-n};q)_r   (q^{-m};q)_r  }{(q^{-a};q)_r (q;q)_r } q^{(n+m-a)r}= \frac{(q^{m-a};q)_n}{(q^{-a};q)_n}.
\end{equation}
We have the following formula for $R_q(n,m,r)$ by Landsberg \cite{landsberg1893}:
\begin{equation}\label{landform}
R_q(n,m,r)=\frac{(-1)^r  (q^{-n};q)_r   (q^{-m};q)_r  q^{(n+m)r-\binom{r}{2}}}{(q;q)_r}.
\end{equation}
By expressing the $r$-th summand of \eqref{chuapp} as
\begin{equation}
\begin{split}
&\frac{(-1)^r  (q^{-n};q)_r (q^{-m};q)_r  q^{(n+m)r-\binom{r}{2}}}{(q;q)_r} \cdot \frac{q^{-ar+\binom{r}{2}}}{(-1)^r (q^{-a};q)_r} \\
&\qquad = R_q(n,m,r) \cdot \frac{q^{-ar+\binom{r}{2}}}{(-1)^r (q^{-a};q)_r},
\end{split}
\end{equation}
we obtain that
\begin{equation} \label{chunotfinal}
\sum_{r \ge 0} R_q(n,m,r)  \frac{q^{-ar+\binom{r}{2}} (-1)^r}{(q^{-a};q)_r}  =\frac{(q^{m-a};q)_n}{(q^{-a};q)_n}.
\end{equation}
The proof is concluded by applying to \eqref{chunotfinal} the simple identity
\begin{equation}
(q^{-x};q)_y  = (-1)^y q^{\binom{y}{2}-xy} \frac{(q;q)_x}{(q;q)_{x-y}}
\end{equation}
with $(x,y) \in \{ (a,n), (a-m,n), (a,r)\}$.
\end{proof}
We now state our main combinatorial identity needed for computing the dimension. Let $k$ be a positive integer. We define the following family of functions.
\begin{equation} \label{fkdef}
f_{k,q}\Big(a;\substack{n_1,\ldots , n_k\\m_1,\ldots,m_k}\Big)=\sum_A \left(q;q\right)_{a-\mathrm{rk}A},
\end{equation}
where $\{ n_i\}_{i=1}^{k}, \{ m_j \}_{j=1}^{k}$ are sequences of non-negative integers, $a$ is an integer such that
\begin{equation}\label{srange}
a \ge \max \{ \sum_{j=1}^{i} n_j + \sum_{j=i}^{k} m_j \mid 1 \le i \le k \}
\end{equation}
and the sum is over all matrices $A\in M_{\sum\limits_{i=1}^k n_i \times\sum\limits_{j=1}^k m_j}(\bff)$ of the form
\begin{equation} \label{nonsqrblocks1}
A=\begin{pmatrix}
Y_{1,1} & Y_{1,2}&\cdots&Y_{1,k} \\
0&Y_{2,2}&\cdots&Y_{2,k}\\
\vdots&\vdots&&\vdots\\
0 & 0&\cdots &Y_{k,k}
\end{pmatrix},
\end{equation}
where $Y_{i,j}\in M_{n_i,m_j}(\bff)$ for all $1\leq i\leq j\leq k$.

\begin{prop} \label{fkprop}
	Let $k \ge 1$. For any sequences of non-negative integers, $\{ n_i\}_{i=1}^{k}$ and $\{ m_j \}_{j=1}^{k}$, and for any integer $a$ satisfying \eqref{srange}, we have
	\begin{equation} \label{fkiden}
	f_{k,q}\Big(a;\substack{n_1,\ldots , n_k\\m_1,\ldots,m_k}\Big)=q^{\sum\limits_{1\leq i\leq j\leq k}n_i m_j}\cdot \frac{\prod\limits_{i=0}^{k}\left(q;q\right)_{a-\sum\limits_{j=1}^{k-i}n_j-\sum\limits_{j=k-i+1}^{k}m_j}}{\prod\limits_{i=1}^{k}\left(q;q\right)_{a-\sum\limits_{j=1}^{k-i+1}n_j-\sum\limits_{j=k-i+1}^{k}m_j}}.
	\end{equation}
\end{prop}
\begin{proof}
 We will use the following notation:
\begin{equation}
I_{m,n} = \begin{pmatrix}
I_m&0\\0&0_{n-m}
\end{pmatrix}, \quad (m\leq n).
\end{equation}
We prove the proposition by induction on $k$. Let $k=1$. Then \begin{equation}
	f_{1,q}\Big(a;\substack{n\\m}\Big)=\sum_{A\in M_{n\times m}(\bff)} \left(q;q\right)_{a-\mathrm{rk}A}=\sum_{r\geq 0} R_q(n,m,r) \left(q;q\right)_{a-r}.
	\end{equation}
	By Lemma \ref{rnmrchu} we find that
	\begin{equation}
	f_1\Big(a;\substack{n\\m}\Big)=q^{nm}\frac{(q;q)_{a-n}(q;q)_{a-m}}{(q;q)_{a-n-m}},
	\end{equation}
	as needed. We now perform the induction step, i.e. assume that \eqref{fkiden} holds for $k-1$ in place of $k$, and prove it for $k$. We split the sum defining $f_{k,q}\Big(a;\substack{n_1,\ldots , n_k\\m_1,\ldots,m_k}\Big)$ as follows:
	\begin{equation} \label{splitfk}
	f_{k,q}\Big(a;\substack{n_1,\ldots , n_k\\m_1,\ldots,m_k}\Big) = \sum_{\substack{Y_{i,i}\in M_{n_i\times m_i}(\bff)\\ 1\leq i \leq k}} \sum_{\substack{Y_{i,j}\in M_{n_i\times m_j}(\bff)\\1\leq i<j\leq k}} \left(q;q\right)_{a-\mathrm{rk}A}.
	\end{equation}
	In the inner sum of \eqref{splitfk} the ranks of $Y_{i,i}$ are fixed for all $1\leq i \leq k$, so we set $r_i=\mathrm{rk}(Y_{i,i})$. There exist invertible matrices $E_i,C_{i}$ such that $Y_{i,i}=E_i I_{r_i,n}C_{i}$, for all $1\leq i \leq k$.
	So, one can write $A$ in the inner sum of \eqref{splitfk} as $\mathrm{diag}\left(E_1,\ldots,E_{k}\right)\cdot\widetilde{A}\cdot	\mathrm{diag}\left(C_1,\ldots,C_{k}\right),$ where
	\begin{equation}\label{atildematrix}
	\widetilde{A}=
	\begin{pmatrix}
	I_{r_1,n} & \widetilde{Y}_{1,2}& \cdots& \widetilde{Y}_{1,k} \\
	0 & I_{r_2,n}& \cdots &\widetilde{Y}_{2,k} \\
	\vdots&\vdots&\ddots&\vdots\\
	0 &0& \cdots &I_{r_{k},n}\\
	\end{pmatrix}
	\end{equation}
	and $\widetilde{Y}_{i,j}=E_{i}^{-1}Y_{i,j}C_{j}^{-1}$ for all $1\leq i<j\leq k$.
	Together with the fact that rank is invariant under elementary operations,  \eqref{splitfk} becomes
	\begin{equation} \label{splitfkrk}
	f_{k,q}\Big(a;\substack{n_1,\ldots , n_k\\m_1,\ldots,m_k}\Big) = \sum_{\substack{\forall 1\leq i \leq k: \\ r_i \ge 0}} \prod_{i=1}^{k} R_q(n_i,m_i,r_i) \sum_{\widetilde{A}} \left(q;q\right)_{a-\mathrm{rk}\widetilde{A}},
	\end{equation}
	where the inner sum is over matrices $\widetilde{A}$ of the form \eqref{atildematrix}. We can use Gaussian elimination operations on $\widetilde{Y}_{i,j}$ for all $1\leq i< j\leq k$ (which do not affect the rank of $\widetilde{A}$) as follows: the first $r_i$ rows of each $\widetilde{Y}_{i,j}$ are being canceled by the pivot elements in $I_{r_{i},n}$ (using elementary row operations) and the first $r_j$ columns of each $\widetilde{Y}_{i,j}$ are being canceled by the pivot elements in $I_{r_j,n}$ (using elementary column operations). Formally, the composition of these elementary operations maps the sequence of matrices $\{ \widetilde{Y}_{i,j} \}_{1\le i<j\le k}$ $\bff$-linearly to a sequence of matrices
	\begin{equation}\label{zijafterel}
	\{ \widehat{\widetilde{Y}}_{i,j} = \begin{pmatrix}
	0&0\\0&Z_{i,j}
	\end{pmatrix} \}_{1 \le i <j \le k},
	\end{equation}
	where $Z_{i,j}\in M_{(n-r_i)\times(n-r_{j})}(\bff)$. This linear map is a projection by construction. Its kernel is of size $q^{\sum_{t=1}^{k-1}r_t \sum_{\ell=t+1}^{k} m_{\ell}+\sum_{t=2}^{k}r_{t} \sum_{\ell=1}^{t-1}(n_{\ell}-r_{\ell})}$. The dimension of the kernel corresponds to the number of elements which we canceled. Equation \eqref{splitfkrk} becomes
	\begin{equation} \label{splitfkrkel}
	\begin{split}
	f_{k,q}\Big(a;\substack{n_1,\ldots , n_k\\m_1,\ldots,m_k}\Big) &= \sum_{\substack{\forall 1 \le i \le k: \\ r_i\geq 0}} \prod_{i=1}^{k} R_q(n_i,m_i,r_i) q^{\sum\limits_{t=1}^{k-1}r_t\sum\limits_{\ell=t+1}^k m_\ell+\sum\limits_{t=2}^k r_t\sum\limits_{\ell=1}^{t-1}\left(n_\ell-r_\ell\right)}\\
	&\qquad \cdot \sum_{\widehat{\widetilde{A}}} \left(q;q\right)_{a-\mathrm{rk}\widehat{\widetilde{A}}},
	\end{split}
	\end{equation}
	where the inner sum is over matrices of the form
	\begin{equation}
	\widehat{\widetilde{A}}=
	\begin{pmatrix}
	I_{r_1,n} & \widehat{\widetilde{Y}}_{1,2}& \cdots& \widehat{\widetilde{Y}}_{1,k} \\
	0 & I_{r_2,n}& \cdots &\widehat{\widetilde{Y}}_{2,k} \\
	\vdots&\vdots&\ddots&\vdots\\
	0 &0& \cdots &I_{r_{k},n}\\
	\end{pmatrix},
	\end{equation}
and $\widehat{\widetilde{Y}}_{i,j}$ are of the form defined in \eqref{zijafterel}.
Note that $\mathrm{rk}\widehat{\widetilde{A}}=\sum_{j=1}^k r_j+\mathrm{rk}Z,$
where
\begin{equation}
Z=
\begin{pmatrix}
Z_{1,2} & \cdots& Z_{1,k} \\
\vdots&\ddots&\vdots\\
0 &\cdots &Z_{k-1,k}\\
\end{pmatrix}.
\end{equation}
Hence, from \eqref{splitfkrkel} we obtain the following recursive relation:
\begin{equation} \label{splitfkrkrec}
\begin{split}
f_{k,q}\Big(a;\substack{n_1,\ldots , n_k\\m_1,\ldots,m_k}\Big) = \sum_{\substack{\forall 1 \le i \le k: \\r_i\geq 0}}& R_q(n_i,m_i,r_i) q^{\sum\limits_{t=1}^{k-1}r_t\sum\limits_{\ell=t+1}^k m_\ell+\sum\limits_{t=2}^k r_t\sum\limits_{\ell=1}^{t-1}\left(n_\ell-r_\ell\right)} \\ &\cdot f_{k-1,q}\Big(a-\sum\limits_{j=1}^k r_j;\substack{n_1-r_1,\ldots , n_{k-1}-r_{k-1}\\m_2-r_2,\ldots,m_k-r_k}\Big).
\end{split}
\end{equation}
Plugging the induction assumption in \eqref{splitfkrkrec} we get that $f_{k,q}\Big(a;\substack{n_1,\ldots , n_k\\m_1,\ldots,m_k}\Big)$ equals
\begin{equation} \label{beforeclosed}
\begin{split}
\sum_{\substack{\forall 1 \le i \le k: \\r_i\geq 0}}& \prod_{i=1}^{k}R_q(n_i,m_i,r_i) q^{\sum\limits_{t=1}^{k-1}r_t\sum\limits_{\ell=t+1}^k m_\ell+\sum\limits_{t=2}^k r_t\sum\limits_{\ell=1}^{t-1}\left(n_\ell-r_\ell\right)} \\ &\cdot q^{\sum\limits_{1\leq i\leq j\leq k-1}\left(n_i-r_i\right)\cdot\left( m_{j+1}-r_{j+1}\right)}\cdot \frac{\prod\limits_{i=0}^{k-1}\left(q;q\right)_{a-\sum\limits_{j=1}^{k}r_j-\sum\limits_{j=1}^{k-1-i}\left(n_j-r_j\right)-\sum\limits_{j=k-i}^{k-1}\left(m_{j+1}-r_{j+1}\right)}}{\prod\limits_{i=1}^{k-1}\left(q;q\right)_{a-\sum\limits_{j=1}^{k}r_j-\sum\limits_{j=1}^{k-i}\left(n_j-r_j\right)-\sum\limits_{j=k-i}^{k-1}\left(m_{j+1}-r_{j+1}\right)}}.
\end{split}
\end{equation}
Rearranging \eqref{beforeclosed}, we see that the sum over $r_1,\ldots,r_k$ may be written as a product over $k$ sums, where the $i$-th sum is over $r_i$:
\begin{equation}\label{beforeclosed2}
\begin{split}
f_{k,q}\Big(a;\substack{n_1,\ldots , n_k\\m_1,\ldots,m_k}\Big) =& \frac{q^{\sum\limits_{1\leq i\leq j\leq k-1}n_i m_{j+1}}}{\prod\limits_{i=1}^{k-1}\left(q;q\right)_{a-\sum\limits_{j=1}^{k-i}n_j-\sum\limits_{j=k-i}^{k-1}m_{j+1}}} \\ &\cdot {} \prod\limits_{i=1}^{k}\Big(\sum_{r_i \ge 0} R_q(n_{i},m_{i},r_{i})\left(q;q\right)_{a-r_{i}-\sum\limits_{j=1}^{i-1}n_j-\sum\limits_{j=i}^{k-1}m_{j+1}}\Big).
\end{split}
\end{equation}
Using Lemma \ref{rnmrchu} we substitute each inner sum of \eqref{beforeclosed2} with
\begin{equation}
q^{n_{i}\cdot m_{i}}\frac{(q;q)_{a-\sum\limits_{j=1}^{i}n_j-\sum\limits_{j=i}^{k-1}m_{j+1}}(q;q)_{s-\sum\limits_{j=1}^{i-1}n_j-\sum\limits_{j=i-1}^{k-1}m_{j+1}}}{(q;q)_{a-\sum\limits_{j=1}^{i}n_j-\sum\limits_{j=i-1}^{k-1}m_{j+1}}},
\end{equation}
and by simplifying we complete the induction step and obtain the desired identity.
\end{proof}
\begin{Rem}
Solomon \cite{Solomon1990} proved a relation between the following two quantities: the number of placements of $k$ non-attacking rooks on a $n \times n$ chessboard, counted  with certain weights depending on $q$, and the number of matrices in $M_{n \times n}(\bff)$ of rank $k$. Haglund generalized Solomon's result to any "Ferrers' board" \cite[Thm.~1]{haglund1998}, which means that the number of matrices of the form \eqref{nonsqrblocks1} over $\bff$ of rank $k$ is related to the $q$-rook polynomial $R_k(B,q)$, where $B$ is a certain Ferrers' board associated with \eqref{nonsqrblocks1}. For the definition of a Ferrers' board and $R_k(B,q)$, see the introduction to the paper by Garsia and Remmel \cite{garsia1986}. In particular, Proposition \ref{fkprop} may be deduced from a result of Garcia and Remmel on $q$-rook polynomials, see \cite[Cor.~2]{haglund1998}. Our proof of Proposition \ref{fkprop} is direct and so we believe it is more accessible. More importantly, the ideas used in the proof reappear in the proofs of Theorem \ref{int_dim} and Theorem \ref{int_char_thm}.
\end{Rem}

\subsection{Arithmetic properties of certain polynomials}
For any $d$ dividing $n$ and any $k \ge 2$, let
\begin{equation}\label{aknddef}
a_{k;n,d}(x)=\frac{x^{d}-1}{x^{n}-1} \sum_{m: \, d \mid m \mid n} \mu\left(\frac{m}{d} \right) (-1)^{k(n- \frac{n}{m})} x^{(k-2)\frac{n}{2} (\frac{n}{m}-1)} \in \mathbb{Q}(x),
\end{equation}
where $\mu: \mathbb{N}_{>0} \to \mathbb{C}$ is the M\"obius function, defined by
\begin{equation}
\mu(n) = \begin{cases} 0 & \text{ if } p^2 \mid n \text{ for some prime }p, \\ (-1)^m & \text{ if }n=p_1 p_2 \ldots p_m, \, p_i \text{ are distinct primes}. \end{cases}
\end{equation}
We recall the following properties of $\mu$ \cite[Ch.~2]{ireland1990}.
\begin{itemize}
	\item The divisor sum $\sum_{d \mid n} \mu(d)$ is given by
	\begin{equation} \label{mobius_prop}
	\sum_{d \mid n} \mu(d) = \delta_{1,n}.
	\end{equation}
	\item The M\"obius function is multiplicative.
\end{itemize}
\begin{seclem}\label{lemand} \
	Let $k \ge 2$. The following hold.
	\begin{itemize}
		\item[(I)] For any $d\mid n$, $a_{k;n,d}(x)$ is a polynomial in $\mathbb{Z}[x]$. Furthermore, in case $d \notin \{ n, \frac{n}{2} \}$, $a_{k;n,d}(x)$ is divisible by $x^d-1$. In the remaining cases we have
		\begin{equation}\label{twocaseslem}
		a_{k;n,d}(x) = \begin{cases} (-1)^{k(n-1)} & \text{if }d=n,\\ \frac{x^{\frac{(k-2)n}{2}}+(-1)^{k+1}}{x^{\frac{n}{2}}+1} & \text{if }d= \frac{n}{2}. \end{cases}
		\end{equation}	
		\item[(II)] If $k > 2$ we have $\deg a_{n,d}(x)= \frac{(n(k-2)-2d)(n-d)}{2d}$, and $a_{n,d}$ has leading coefficient $(-1)^{k(n-\frac{n}{d})}$. If $k=2$, we have $a_{n,d} = \delta_{n,d}$.
		\item[(III)] Assume $k>2$. For any prime power $q$, $a_{k;n,d}(q)$ is a non-zero integer. Its sign equals the sign of $(-1)^{k(n-\frac{n}{d})}$, i.e. it is a positive integer unless $k$ is odd, $n$ is even and $2 \nmid \frac{n}{d}$.
	\end{itemize}
\end{seclem}

\begin{proof}
	We begin by proving the first part of the lemma. If $d \in \{n, \frac{n}{2}\}$, a short calculation reveals that \eqref{twocaseslem} holds. From now on we assume that $d \notin \{n, \frac{n}{2} \}$. We shall show that
	\begin{equation}\label{divpartlem}
	x^n-1 \mid \sum_{m: \, d \mid m \mid n} \mu\left(\frac{m}{d}\right) (-1)^{k(n- \frac{n}{m})} x^{(k-2)\frac{n}{2} (\frac{n}{m}-1)}
	\end{equation}
	in $\mathbb{Q}[x]$, which implies that $a_{k;n,d}(x)$ is a polynomial divisible by $x^d-1$. Gauss's lemma, applied to \eqref{divpartlem}, implies that $a_{k;n,d}(x)\in\mathbb{Z}[x]$. We now prove \eqref{divpartlem}.
	
	Let $z$ be a root of unity of order dividing $n$. Assume first that $n$ is odd or that $k$ is even. Then for all $m\mid n$ we have
	\begin{equation}
	z^{(k-2)\frac{n}{2}(\frac{n}{m}-1)} = (z^n)^{(k-2)\frac{\frac{n}{m}-1}{2}}= 1.
	\end{equation}
	Hence, using \eqref{mobius_prop},
	\begin{equation}\label{oddiden}
	\sum_{m: \, d \mid m \mid n} \mu\left(\frac{m}{d} \right) (-1)^{k(n- \frac{n}{m})} z^{(k-2)\frac{n}{2} (\frac{n}{m}-1)} = \sum_{m: \, d \mid m \mid n} \mu\left(\frac{m}{d} \right)=\sum_{a: \, a \mid \frac{n}{d}} \mu(a) = \delta_{d,n} = 0.
	\end{equation}	
	Now we assume instead that $n$ is even and $k$ is odd. We are led to consider two cases.
	\begin{itemize}
		\item If $z^{\frac{n}{2}} = -1$ then for all $m\mid n$ we have,
		\begin{equation}
		z^{(k-2)\frac{n}{2}(\frac{n}{m}-1)} = (-1)^{\frac{n}{m}-1}.
		\end{equation}
		Hence,  using \eqref{mobius_prop},
		\begin{equation}\label{eveniden1}
		\sum_{m: \, d \mid m \mid n} \mu\left(\frac{m}{d}\right) (-1)^{k(n- \frac{n}{m})} z^{(k-2)\frac{n}{2} (\frac{n}{m}-1)} = -\sum_{m: \, d \mid m \mid n} \mu\left(\frac{m}{d}\right) = -\sum_{a \mid \frac{n}{d}} \mu(a) = -\delta_{d,n} = 0.
		\end{equation}
		\item If $z^{\frac{n}{2}}=1$ then for all $m\mid n$ we have,
		\begin{equation}
		z^{(k-2)\frac{n}{2}(\frac{n}{m}-1)} = 1.
		\end{equation}
		Hence,
		\begin{equation} \label{eveniden2}
		\begin{split}
		\sum_{m: \, d \mid m \mid n} \mu\left(\frac{m}{d}\right) (-1)^{k(n- \frac{n}{m})} z^{(k-2)\frac{n}{2} (\frac{n}{m}-1)} &= \sum_{m: \, d \mid m \mid n} \mu\left(\frac{m}{d}\right) (-1)^{\frac{n}{m}} = \sum_{a \mid \frac{n}{d}} \mu(a) (-1)^{\frac{n}{ad}} \\
		&= \sum_{\substack{a \mid \frac{n}{d} \\ 2 \mid \frac{n}{ad} }} \mu(a)  - \sum_{\substack{a \mid \frac{n}{d} \\ 2 \nmid \frac{n}{ad} }} \mu(a) \\
		& = \begin{cases} 0-\sum_{a \mid \frac{n}{d}}\mu(a) &\text{ if } 2 \nmid \frac{n}{d} \\ \sum_{a \mid \frac{n}{2d}} \mu(a) - \sum_{\substack{a \mid \frac{n}{d}\\ 2 \mid a}} \mu(2 \cdot \frac{a}{2}) & \text{ if } 2 \mid \frac{n}{d}, 4 \nmid \frac{n}{d} \\ \sum_{a \mid \frac{n}{2d}} \mu(a) - \sum_{\substack{a \mid \frac{n}{d} \\ 2 \nmid \frac{n}{ad} }} \mu(4\cdot \frac{a}{4}) & \text{ if } 4 \mid \frac{n}{d}\end{cases} \\
		&=  \begin{cases} -\delta_{d,n}& \text{ if }2 \nmid \frac{n}{d} \\ \delta_{2d,n} - \mu(2) \delta_{2d,n}  & \text{ if }2 \mid \frac{n}{d}, 4 \nmid \frac{n}{d} \\ \delta_{2d,n} & \text{ if }4 \mid \frac{n}{d}\end{cases} \\&=0.
		\end{split}
		\end{equation}
	\end{itemize}
	Equations \eqref{oddiden}, \eqref{eveniden1} and \eqref{eveniden2} show that the RHS of \eqref{divpartlem} vanishes on each root of the separable polynomial $x^n-1$, which establishes \eqref{divpartlem}. This concludes the proof of the first part of the lemma.
	
	The second part of the lemma for $k>2$ follows by noticing that the numerator of $a_{k;n,d}(x)$ has degree $d + (k-2)\frac{n}{2}(\frac{n}{d}-1)$ (arising from the term corresponding to $m=d$) and leading coefficient equal to $(-1)^{k(n-\frac{n}{d})}$, while the denominator of $a_{k;n,d}(x)$ has degree $n$ and leading coefficient equal to $1$.
	
	When $k=2$, all terms in the sum in \eqref{aknddef} are constants, and we have
	\begin{equation}
	a_{2;n,d}(x)=\frac{x^{d}-1}{x^{n}-1} \sum_{m: \, d \mid m \mid n} \mu\left(\frac{m}{d} \right) = \frac{x^d-1}{x^n-1}\delta_{n,d} = \delta_{n,d}.
	\end{equation}
	We now turn to the third part of the lemma. Since $a_{k;n,d}(x)$ has integer coefficients, $a_{k;n,d}(q)$ is an integer. We now determine its sign when $k>2$, and in particular show that it is non-zero.
	
	Since $q^d-1$, $q^n-1$, $q^{\frac{n}{2}}$ are positive, we deal with the expression
	\begin{equation}
	\begin{split}
	\widetilde{a}_{k;n,d}(q) :&= \frac{q^n-1}{q^d-1}q^{(k-2)\frac{n}{2}}\cdot a_{k;n,d}(q) \\&= \sum_{m: \, d\mid m \mid n} \mu\left(\frac{m}{d}\right) (-1)^{k(n-\frac{n}{m})} (q^{(k-2)\frac{n}{2}})^{\frac{n}{m}}\\&= \sum_{a \mid \frac{n}{d}} \mu(a) (-1)^{k(n-\frac{n}{ad})} (q^{(k-2)\frac{n}{2}})^{\frac{n}{ad}},
	\end{split}
	\end{equation}
	whose sign is the same as the sign of $a_{k;n,d}(q)$.
	If $d=n$ then
	\begin{equation}
	(-1)^{k(n-\frac{n}{d})} \widetilde{a}_{k;n,d}(q) = q^{(k-2)\frac{n}{2}} >0.
	\end{equation}
	If $d=\frac{n}{2}$ then
	\begin{equation}
	(-1)^{k(n-\frac{n}{d})} \widetilde{a}_{k;n,d}(q) =  (q^{(k-2)\frac{n}{2}})^{2} +(-1)^{k+1} q^{(k-2)\frac{n}{2}} > 0.
	\end{equation}
	If $\frac{n}{d} \ge 3$, we set $t = q^{(k-2)\frac{n}{2}}$. Then, $t \ge 2^{\frac{3}{2}} >2$ and
	\begin{equation}
	\begin{split}
	(-1)^{k(n-\frac{n}{d})} \widetilde{a}_{k;n,d}(q) &\ge (q^{(k-2)\frac{n}{2}})^{\frac{n}{d}} - \sum_{1 \le i \le \frac{n}{2d}} (q^{(k-2)\frac{n}{2}})^i \ge (q^{(k-2)\frac{n}{2}})^{\frac{n}{d}} - \frac{(q^{(k-2)\frac{n}{2}})^{\frac{n}{2d}}}{1-q^{-(k-2)\frac{n}{2}}} \\
	& = (q^{(k-2)\frac{n}{2}})^{\frac{n}{2d}} \left( (q^{(k-2)\frac{n}{2}})^{\frac{n}{2d}} - \frac{1}{1-q^{-(k-2)\frac{n}{2}}}\right)\\
	& \ge (q^{(k-2)\frac{n}{2}})^{\frac{n}{2d}} \left( (q^{(k-2)\frac{n}{2}})^{1.5} - \frac{1}{1-q^{-(k-2)\frac{n}{2}}}\right) \\
	& = \frac{(q^{(k-2)\frac{n}{2}})^{\frac{n}{2d}}}{1-q^{-(k-2)\frac{n}{2}}} \left(t^{\frac{1}{2}}(t-1) - 1\right) > 0.
	\end{split}
	\end{equation}

\end{proof}

\begin{Rem}
	The polynomials $a_{k;n,q}(x)$ may be expressed using the necklace polynomials (see Moreau \cite{moreau1872}), defined by 
	\begin{equation}
	M_n(x) = \frac{1}{n} \sum_{d \mid n} \mu(d) x^\frac{n}{d}.
	\end{equation}
	Indeed,
	\begin{equation}
	a_{k;n,d}(x) = \frac{x^d-1}{x^n-1} \cdot \left(\frac{(-1)^n}{x^\frac{n}{2}} \right)^{k-2} \cdot M_{\frac{n}{d}}\left(\left(-x^\frac{n}{2} \right)^{k-2}\right).
	\end{equation}
\end{Rem}

\section{Calculation of the Dimension of $\pi_{k,N,\psi}$}\label{dim_sec}

Here we prove Theorem \ref{int_dim}. Given $U\in N$, we write it in the notation of \eqref{udef}.
From \eqref{projform},
\begin{equation} \label{dim_first}
	\mathrm{dim}\left(\pi_{k,N,\psi}\right)=\frac{1}{|N|}\sum_{\unimat \in N}\Theta_\theta (\unimat)\overline{\psi }(\unimat)= \frac{1}{q^{\binom{k}{2}n^{2}}}\sum_{\unimat \in N}\Theta_\theta\left(\unimat\right)\overline{\psi}\left(\unimat\right).
\end{equation}
By Corollary \ref{greencor}, the value $\Theta_\theta(\unimat)$ is determined by $\mathrm{dim}_{\bff_{kn}}\ker (\unimat-I)$ which is in turn determined by $\mathrm{rank}_{\bff_{kn}}(\unimat-I)$.
Therefore, we will start by splitting the sum in \eqref{dim_first} by the $X_{i,i}$, $1\leq i\leq k-1$.
\begin{equation}
	\mathrm{dim}\left(\pi_{k,N,\psi}\right)=\frac{1}{q^{\binom{k}{2}n^{2}}} \sum_{\substack{X_{i,i}\in M_n(\bff)\\1\leq i \leq k-1}}
	\sum_{\substack{X_{i,j}\in M_n\left(\bff\right)\\1\le i< j \le k-1}}\Theta_\theta\left(\unimat\right)\overline{\psi}\left(\unimat\right).
\end{equation}
The character $\psi\left(\unimat\right)=\psi\left(X_{1,1},\ldots,X_{k-1,k-1}\right)$ is determined by the traces of $X_{i,i}$, $1\leq i\leq k-1$. Hence,
\begin{equation} \label{splitsumdim}
\mathrm{dim}\left(\pi_{k,N,\psi}\right)=\frac{1}{q^{\binom{k}{2}n^{2}}} \sum_{\substack{X_{i,i}\in M_n(\bff)\\1\leq i \leq k-1}}   \overline{\psi}\left(\unimat\right)
\sum_{\substack{X_{i,j}\in M_n\left(\bff\right)\\1 \le i < j \le k-1}}\Theta_\theta\left(\unimat\right).
\end{equation}
In the inner sum of \eqref{splitsumdim} set $r_i=\mathrm{rk}\left(X_{i,i}\right)$ for $1\leq i \leq k-1$. There exist invertible matrices $E_i,C_{i+1}$ such that $X_{i,i}=E_i I_{r_i,n}C_{i+1}$.
So, one can write $\unimat$ in the inner sum of \eqref{splitsumdim} as $I_{kn}$ plus
\begin{equation} \label{elementaryop1}
	\mathrm{diag}\left(E_1,\ldots,E_{k-1},I_n\right)
	\begin{pmatrix}
	0 & I_{r_1,n} &  \cdots&\widetilde{X}_{1,k-2}&\widetilde{X}_{1,k-1} \\
	0 & 0 &  \cdots &\widetilde{X}_{2,k-2}&\widetilde{X}_{2,k-1} \\
	\vdots&\vdots&&\vdots&\vdots\\
	0 & 0 & \cdots & 0 &I_{r_{k-1},n}\\
	0 & 0 & \cdots & 0 &0
	\end{pmatrix}
	\mathrm{diag}\left(I_n,C_2,\ldots,C_{k}\right),
\end{equation}
where $\widetilde{X}_{i,j}=E_{i}^{-1}X_{i,j}C_{j+1}^{-1}$ for all $1 \le i< j \le k-1$.
 Together with the fact that rank is invariant under elementary operations, we now have
\begin{equation}\label{dim0}
\begin{split}
\mathrm{dim}\left(\pi_{k,N,\psi}\right)=\frac{1}{q^{\binom{k}{2}n^{2}}}  &\sum_{\substack{X_{i,i}\in M_n(\bff)\\1\leq i \leq k-1}} \overline{\psi}\left(\unimat\right)
\\
&\cdot\sum_{\substack{\widetilde{X}_{i,j}\in M_n\left(\bff\right)\\1 \le i < j \le k-1}}\Theta_\theta \left(I_{kn}+
			\begin{pmatrix}
			0 & I_{r_1,n} &  \cdots&\widetilde{X}_{1,k-2}&\widetilde{X}_{1,k-1} \\
			0 & 0 &  \cdots &\widetilde{X}_{2,k-2}&\widetilde{X}_{2,k-1} \\
			\vdots&\vdots&&\vdots&\vdots\\
			0 & 0 & \cdots & 0 &I_{r_{k-1},n}\\
			0 & 0 & \cdots & 0 &0
			\end{pmatrix}\right).
\end{split}
\end{equation}
As in the proof of Proposition \ref{fkprop}, we can use Gaussian elimination operations on $\widetilde{X}_{i,j}$ for all $1 \le i <j \le k-1$ (which do not affect the rank nor dimension of the kernel of the matrix minus $I_{kn}$, and the number of Jordan blocks is not affected as well) in such a way that the sequence of matrices $\{ \widetilde{X}_{i,j} \}_{1\le i<j\le k-1}$ is mapped $\bff$-linearly to a sequence of matrices
\begin{equation}
\{ \widehat{\widetilde{X}}_{i,j} = \begin{pmatrix}
0&0\\0&Y_{i,j}
\end{pmatrix} \}_{1 \le i <j \le k-1},
\end{equation}
where $Y_{i,j}\in M_{(n-r_i)\times(n-r_{j})}(\bff)$. The kernel of this mapping is of size $q^{\sum_{i=1}^{k-2}r_i(k-i-1)n+\sum_{i=2}^{k-1}r_i\sum_{j=1}^{i-1}(n-r_j)}$. The dimension of the kernel corresponds to the number of elements which we cancel. Equation \eqref{dim0} becomes
\begin{equation}\label{dim1}
\begin{split}
\mathrm{dim}\left(\pi_{k,N,\psi}\right)=\frac{1}{q^{\binom{k}{2}n^{2}}}  &\sum_{\substack{X_{i,i}\in M_n(\bff)\\1\leq i \leq k-1}} \overline{\psi}\left(\unimat\right)
q^{\sum\limits_{i=1}^{k-2}r_i(k-i-1)n+\sum\limits_{i=2}^{k-1}r_i\sum\limits_{j=1}^{i-1}(n-r_j)}\\
&\cdot\sum_{\substack{Y_{i,j}\in M_{(n-r_i)\times(n-r_{j})}(\bff)\\ 1 \le i <j \le k-1}}\Theta_\theta \left(g\right),
\end{split}
\end{equation}
where
\begin{equation}
g= I_{kn}+\begin{pmatrix}
0 & I_{r_1,n} &  \cdots&\widehat{\widetilde{X}}_{1,k-2}&\widehat{\widetilde{X}}_{1,k-1} \\
0 & 0 &  \cdots &\widehat{\widetilde{X}}_{2,k-2}&\widehat{\widetilde{X}}_{2,k-1} \\
\vdots&\vdots&&\vdots&\vdots\\
0 & 0 & \cdots & 0 &I_{r_{k-1},n}\\
0 & 0 & \cdots & 0 &0
\end{pmatrix}
\end{equation}
According to the character formula \eqref{green}, we can calculate $\Theta_\theta(g)$. In this case $m=kn$, $g=s \cdot u$ where $s=I_{kn}$, so $\lambda=1$ and
$$t=\mathrm{dim}\ \mathrm{ker}(g-I)=kn-\mathrm{rk}(g-I)=kn-\sum_{i=1}^{k-1}r_i -\mathrm{rk}A,$$
where
\begin{equation} \label{nonsqrblocks}
A=\begin{pmatrix}
Y_{1,2} & \cdots&Y_{1,k-1} \\
\vdots&&\vdots\\
0 & \cdots &Y_{k-2,k-1}
\end{pmatrix},\quad 1 \le i < j \le k-1.
\end{equation}
So,
\begin{align}
	\Theta_\theta \left(g\right) ={}& (-1)^{kn-1}(1-q)(1-q^2)\cdots(1-q^{kn-\sum\limits_{i=1}^{k-1}r_i -\mathrm{rk}A-1})\nonumber\\
	={}&(-1)^{kn-1}(q;q)_{kn-\sum\limits_{i=1}^{k-1}r_i -\mathrm{rk}A-1}. \label{cuspchar}
\end{align}
Equation \eqref{dim1} can now be written as
\begin{equation} \label{dim2}
\begin{split}
\mathrm{dim}\left(\pi_{k,N,\psi}\right)=\frac{1}{q^{\binom{k}{2}n^{2}}}  &\sum_{\substack{X_{i,i}\in M_n(\bff)\\1\leq i \leq k-1}} \overline{\psi}\left(\unimat\right)
q^{\sum \limits_{i=1}^{k-2}r_i(k-i-1)n+\sum\limits_{i=2}^{k-1}r_i\sum \limits _{j=1}^{i-1}(n-r_j)}\\
&\cdot(-1)^{kn-1}\sum_{A}(q;q)_{kn-\sum \limits_{i=1}^{k-1}r_i -\mathrm{rk}A-1},
\end{split}
\end{equation}
where the inner sum is over all matrices of the form \eqref{nonsqrblocks} and by the definition \eqref{fkdef} it is equal to
\begin{equation}
f_{k-2,q}\Big(kn-\sum_{i=1}^{k-1}r_i-1;\substack{n-r_1,\ldots , n-r_{k-2}\\n-r_2,\ldots,n-r_{k-1}}\Big).
\end{equation}
 By applying Proposition \ref{fkprop} we replace the inner sum in \eqref{dim2} by
\begin{equation}
 q^{\sum\limits_{1\leq i\leq j\leq k-2}(n-r_i)\cdot (n-r_{j+1})}\cdot \frac{\prod\limits_{i=0}^{k-2}\left(q;q\right)_{kn-\sum\limits_{j=1}^{k-1}r_j-1-\sum\limits_{j=1}^{k-2-i}(n-r_j)-\sum\limits_{j=k-i-1}^{k-2}(n-r_{j+1})}}{\prod\limits_{i=1}^{k-2}\left(q;q\right)_{kn-\sum\limits_{j=1}^{k-1}r_j-1-\sum\limits_{j=1}^{k-i-1}(n-r_j)-\sum\limits_{j=k-i-1}^{k-2}(n-r_{j+1})}},
\end{equation}
which equals
\begin{equation}
q^{\sum\limits_{1\leq i\leq j\leq k-2}(n-r_i)\cdot (n-r_{j+1})}\cdot \frac{\prod\limits_{i=1}^{k-1}\left(q;q\right)_{2n-1-r_i}}{\left( (q;q)_{n-1} \right)^{(k-2)}}.
\end{equation}
Now \eqref{dim2} becomes
\begin{equation} \label{sumthenprod}
\mathrm{dim}\left(\pi_{k,N,\psi}\right)=\frac{(-1)^{kn-1}}{\left((q;q)_{n-1}\right)^{\left(k-2\right)}q^{(k-1)n^{2}}}  \sum_{\substack{X_{i,i}\in M_n(\bff)\\1\leq i \leq k-1}}\prod_{i=1}^{k-1} \overline{\psi_{0}}\left(\mathrm{tr}\left(X_{i,i}\right)\right) \left(q;q\right)_{2n-1-r_i}.
\end{equation}
Changing the order of sum and product in \eqref{sumthenprod} we get that
\begin{equation}\label{prasadproducts}
\mathrm{dim}\left(\pi_{k,N,\psi}\right)=\frac{(-1)^{kn-1}}{\left((q;q)_{n-1}\right)^{\left(k-2\right)}q^{(k-1)n^{2}}} \prod_{i=1}^{k-1} \sum_{X_{i,i}\in M_n(\bff)} \overline{\psi_{0}}\left(\mathrm{tr}\left(X_{i,i}\right)\right) \left(q;q\right)_{2n-1-r_i}.
\end{equation}
From Section 5 of \cite{prasad2000}, each inner sum in \eqref{prasadproducts} is equal to
\begin{equation}\label{prasadres}
\sum_{X_{i,i}\in M_n(\bff)} \overline{\psi_{0}}\left(\mathrm{tr}\left(X_{i,i}\right)\right) \left(q;q\right)_{2n-1-r_i} = (-1)^{n} \cdot q^{n^2} \cdot q^{\binom{n}{2}} (q;q)_{n-1}.
\end{equation}
Plugging \eqref{prasadres} in \eqref{prasadproducts}, we obtain
\begin{equation}\label{prasadproducts2}
\mathrm{dim}\left(\pi_{k,N,\psi}\right)= q^{(k-1)\binom{n}{2}} (-1)^{n-1}(q;q)_{n-1} =q^{(k-2)\binom{n}{2}} \frac{|\gl{n}|}{q^n-1},
\end{equation}
as needed. \qed

\section{Calculation of the Character $\Theta_{k,N,\psi}$} \label{char_sec}
In this section we prove Theorem \ref{int_char_thm}. Namely, we calculate $\Theta_{k,N,\psi}$. From now on we will use the following notations:

\begin{equation}
h_{g;\unimat} =  \begin{pmatrix}
		g & X_{1,1} & X_{1,2} & \cdots&X_{1,k-2}&X_{1,k-1} \\
		0 & g & X_{2,2} & \cdots &X_{2,k-2}&X_{2,k-1} \\
		0& 0 & g & \cdots &X_{3,k-2},& X_{3,k-1}\\
		\vdots&\vdots&\vdots&&\vdots&\vdots\\
		0 & 0 & 0& \cdots & g &X_{k-1,k-1}\\
		0 & 0 & 0& \cdots & 0 &g
	\end{pmatrix},
\end{equation}
where $\unimat$ (and so $X_{i,j}$) were defined in \eqref{udef}.
Note that $h_{I_n,U}=U$. We also define
\begin{equation}
\Delta^r\left(g\right)  = \mathrm{diag}\left(g,\ldots,g\right)\in \Delta^r\left(\gl{m}\right), \qquad  g\in \gl{m}.
\end{equation} By definition,
\begin{equation} \label{charbyproj}
\begin{split}
\Theta_{k,N,\psi}\left(g\right)=\mathrm{tr}\left(\pi_{k,N,\psi}\left(g\right)\right)={}& \mathrm{tr}\left(\pi(\Delta^k(g)){\restriction_{V_{k,N,\psi}}}\right)\\={}& \mathrm{tr}\left(\pi(\Delta^k(g))\circ P_{k,N,\psi}\right).
\end{split}
\end{equation}
Substituting \eqref{projform} into \eqref{charbyproj} we have
\begin{equation}\label{thetaexpansion1}
\begin{split}
 	\Theta_{k,N,\psi}\left(g\right)={}& \mathrm{tr}\left(\frac{1}{q^{\binom{k}{2}n^{2}}}\sum_{\unimat \in N}\pi\left[\Delta^k(g)\ \cdot \unimat \right]\overline{\psi}\left(\unimat \right)\right)\\
={}  & \frac{1}{q^{\binom{k}{2}n^{2}}}\sum_{\unimat \in N}\mathrm{tr}\left(\pi\left[\Delta^k(g)\cdot \unimat\right]\right)\overline{\psi}\left(\unimat\right).
\end{split}
\end{equation}
Now we perform the change of variables
\begin{equation}
X_{i,j} \mapsto g^{-1}X_{i,j}, \qquad 1 \le i \le j \le k-1
\end{equation}
in \eqref{thetaexpansion1} and obtain
\begin{equation} \label{charformula}
\Theta_{k,N,\psi}\left(g\right)= \frac{1}{q^{\binom{k}{2}n^{2}}}\sum_{\unimat \in N}\Theta_{\theta}\left(h_{g;\unimat}\right)\overline{\psi}\left(g^{-1}X_{1,1},\ldots,g^{-1}X_{k-1,k-1}\right).
\end{equation}

In parts \S \ref{nonsubsem}, \S \ref{nonsemsub} and \S \ref{semsub} we prove parts (I), (II) and (III) of Theorem \ref{int_char_thm}, respectively.

\subsection{Character at $g = s\cdot u$ such that the semisimple part $s$ does not come from $\bff_n$} \label{nonsubsem}
Let $g = s\cdot u$. Assume that the semisimple part $s$ does not come from $\bff_n$. The semisimple part of $h_{g;U}$ is $\Delta^k(s)$, which also does not come from $\bff_n$. By Theorem \ref{greenthm}, we have $\Theta_\theta\left(h_{g;U}\right)=0$. Hence, by \eqref{charformula} $\Theta_{k,N,\psi}\left(g\right)=0$.
\qed

\subsection{Character calculation at a non-semisimple element} \label{nonsemsub}
Assume that $s$ comes from $\bff_d\subseteq \bff_n$ and $d\mid n$ is minimal. In addition, $d<n$ since $g$ is not semisimple.
Let $\lambda\in \bff_d^*$ be an eigenvalue of $s$ which generates the field $\bff_d$ over $\bff$. We use the notations of \S \ref{conj_class}. Thus, there exist $R\in\gl{n}$ and $\rho$ partition of $n/d$ such that $R^{-1}gR=L_\rho(f)$ and there exists $\Delta^{d^\prime}\left(T\right)\in \gl[\bff_d]{n}$ such that
\begin{equation}
g_{\rho}=\Delta^{d^\prime}\left(T^{-1}\right)R^{-1}gR\Delta^{d^\prime}\left(T\right),
\end{equation}
the analogue of the Jordan form of $g$.
Recall that by Lemma \ref{iso_f_to_fd}, the map $$A\mapsto A_\rho:=A_{\rho,R}=\Delta^{d^\prime}\left(T^{-1}\right)R^{-1}AR\Delta^{d^\prime}\left(T\right)$$ induces an isomorphism. By the notation of \S \ref{cong_arb_mat} we have for each
\begin{equation}
X_{a,b}, \qquad \forall 1 \le a \le b \le k-1,
\end{equation}
the corresponding isomorphism of Lemma \ref{iso_f_to_fd}
\begin{equation*}
X_{a,b}\mapsto
\begin{pmatrix}
\left(x^{(a,b)}_{0,i;j}\right)_{\substack{0\leq i \leq d-1\\ 0\leq j \leq d^\prime-1}}\\
\vdots\\
\left(x^{(a,b)}_{d^\prime-1,i;j}\right)_{\substack{0\leq i \leq d-1\\ 0\leq j \leq d^\prime-1}}
\end{pmatrix}.
\end{equation*}
Note that
\begin{equation}\label{bigiso}
 \Delta^k\left(\Delta^{d^\prime}\left(T^{-1}\right)\right)\Delta^k\left(R^{-1}\right) h_{g;\unimat} \Delta^k\left(R\right)\Delta^k\left( \Delta^{d^\prime}\left(T\right)\right)=h_{g_{\rho};\unimat_{\rho}},
\end{equation}
where $\unimat_{\rho}$ is the element of $N$ with $(X_{a,b})_{\rho}$ instead of $X_{a,b}$. From \eqref{bigiso} we obtain
\begin{equation}
\mathrm{rk}\left( h_{g-\lambda I_n;\unimat} \right)=\mathrm{rk}\left( h_{g_{\rho}-\lambda I_n;\unimat_\rho} \right).
\end{equation}
We prove that $\mathrm{rk}\left( h_{g-\lambda I_n;\unimat} \right)$ (which by Corollary \ref{greencor} determines the value of $\Theta _\theta\left( h_{g;\unimat} \right)$) is independent of $x^{(k-1,k-1)}_{1,0;1}\in\bff_d$.
The matrix $h_{g_\rho-\lambda I_n;\unimat_\rho}$ has the form
\begin{equation}\label{h_in_fd}
h_{g_{\rho}-\lambda I_n;\unimat_\rho}=  \begin{pmatrix}
		g_\rho-\lambda I_n & (X_{1,1})_\rho  & \cdots&(X_{1,k-2})_{\rho}&(X_{1,k-1})_{\rho} \\
		0 & g_\rho-\lambda I_n  & \cdots &(X_{2,k-2})_{\rho}&(X_{2,k-1})_{\rho} \\
		\vdots&\vdots&&\vdots&\vdots\\
		0 & 0 & \cdots & g_\rho-\lambda I_n &(X_{k-1,k-1})_{\rho}\\
		0 & 0 & \cdots & 0 &\boxed{g_\rho-\lambda I_n}
	\end{pmatrix}.
\end{equation}
Consider the boxed block in \eqref{h_in_fd}. The $2d\times 2d$ upper left block of the boxed matrix $g_\rho-\lambda I_n$ has the form
\begin{equation} \label{boxed1}
\begin{pmatrix}
0&&&\boxed{1}&&&\\&\lambda^q-\lambda&&&1&&\\&&\ddots&&&1&\\&&&\lambda^{q^{d-1}}-\lambda&&&1\\					&&&0&&&\\&&&&\lambda^q-\lambda&&\\&&&&&\ddots&\\&&&&&&\lambda^{q^{d-1}}-\lambda
\end{pmatrix}
\end{equation}
Let $Z:=X_{k-1,k-1}$, $Z_{\rho} := (X_{k-1,k-1})_{\rho}$ and $z_{m,i;j}:=x^{(k-1,k-1)}_{m,i;j}$. One can eliminate the $(d+1)$-th column in $Z_\rho$ by the boxed $1$ from \eqref{boxed1}, i.e. all the elements $\left\{z_{m,i;1}\right\}_{\substack{0\leq i\leq d-1 \\ 0\leq m\leq d^\prime-1}}$. In particular, $z_{1,0;1} = x^{(k-1,k-1)}_{1,0;1}$ is eliminated.
Now, by Lemma \ref{lem_tr_conj}, \eqref{charformula} can be written as
\begin{equation} \label{char_form_uni2_companion}
\begin{split}
	\Theta _{k,N,\psi }(g)  ={}&\frac{1}{q^ { \binom{k}{2} n^2 }} \sum \limits _{\unimat\in N} \Theta _\theta
\left( h_{g;\unimat} \right)\cdot  \prod_{i=1}^{k-2}\overline{\psi_0}\left(g^{-1}X_{i,i}\right) \\
&\quad \cdot \overline \psi_0 \left(\mathrm{Tr}_{\bff_{d}/\bff}\left(\lambda ^{-1} \cdot \sum \limits _{m=0} ^{d^\prime-1} z_{m,0;m}\right)+\mathrm{tr}\left(D^{-2}\alpha\left(g,D^{-1},Z_\rho\right)\right)
\right).
\end{split}
\end{equation}
By Lemma \ref{iso_f_to_fd}, going over $Z\in M_n(\bff)$ is equivalent to going over $\left( z_{m,i;j}\right)_{\substack{0\leq i \leq d-1\\ 0\leq j,m \leq d^\prime-1}}$, $z_{m,i;j}\in\bff_d$. We have just shown that $\Theta _\theta\left( h_{g;\unimat} \right)$ is independent of $z_{1,0;1}$, and by Lemma \ref{lem_tr_conj} $\mathrm{tr}\left(D^{-2}\alpha\left(g,D^{-1},Z_\rho\right)\right)$ in \eqref{char_form_uni2_companion} is also independent of $z_{1,0;1}$. Thus, we may write \eqref{char_form_uni2_companion} as the following double sum, where the inner sum is over $z_{1,0;1}$ and the outer sum is over the rest of the coordinates of $U$:
\begin{equation} \label{char_form_uni2_companion2}
\begin{split}
	\Theta _{k,N,\psi }(g)  ={}&\frac{1}{q^ { \binom{k}{2} n^2 }} \sum \limits _{ \substack{X_{i,j}\in N, (i,j) \neq (k-1,k-1) \\ z_{m,i;j} \in \bff_d, (m,i,j) \neq (1,0,1)}} \Theta _\theta
\left( h_{g;\unimat} \right)\cdot  \prod_{i=1}^{k-2}\overline{\psi_0}\left(g^{-1}X_{i,i}\right) \\
&\quad \cdot  \overline \psi_0 \left(\mathrm{tr}\left(D^{-2}\alpha\left(g,D^{-1},Z_\rho\right)\right) \right)  \cdot \overline \psi_0 \left(\mathrm{Tr}_{\bff_{d}/\bff}\left(\lambda ^{-1} \cdot \sum\limits _{\substack{0 \le m \le d'-1 \\ m \neq 1}} z_{m,0;m}\right)\right) \\
&\quad \cdot \sum_{z_{1,0;1} \in \bff_d} \overline{\psi_0} \left(\mathrm{Tr}_{\bff_{d}/\bff}\left(\lambda ^{-1} \cdot z_{1,0;1}\right)\right).
\end{split}
\end{equation}
Since $\overline \psi_0\circ \mathrm{Tr}_{\bff_{d}/\bff}$ is a nontrivial character, we have
\begin{equation}
	\sum \limits _{z_{1,0;1}\in \bff_d}\overline \psi_0 \left(\mathrm{Tr}_{\bff_{d}/\bff}\left(\lambda ^{-1} \cdot z_{1,0;1}\right)
	\right)=0.
\end{equation}
Thus, $\Theta _{k,N ,\psi }(g) =0$.
\qed 

\subsection{Character calculation at a semisimple element} \label{semsub}
Here we will use \eqref{charformula} to calculate the value of $\Theta_{k,N,\psi}(g)$ for $g=s$ where $s$ is semisimple element which comes from a subfield of $\bff_{n}$ ($u=I_n$). Again, we use the notations of \S \ref{conj_class}. Thus, there exist $R\in\gl{n}$, $\rho$ partition of $n/d$ and $\Delta^{d^\prime}\left(T\right)\in \gl[\bff_d]{n}$ such that
\begin{equation} \label{defrts}
s_{\rho}=\Delta^{d^\prime}\left(T^{-1}\right)R^{-1}sR\Delta^{d^\prime}\left(T\right),
\end{equation}
the analogue of the Jordan form of $s$. We also use the notations of \S \ref{cong_arb_mat}, and in particular define $(X_{a,b})_{\rho}$ as in \S \ref{nonsemsub}.

Let $\lambda\in \bff_n^*$ be an eigenvalue of $s$.
If $\lambda\in\bff^*$ then $s=\lambda I$, and we have by \eqref{charformula}
\begin{equation}
	\Theta _{k,N,\psi }(\lambda I)  ={}\frac{1}{q^ { \binom{k}{2} n^2 }} \sum \limits _{\unimat\in N} \Theta _\theta
\left( h_{\lambda I;\unimat} \right)  \overline{\psi}\left(\lambda^{-1}X_{1,1}, \ldots,\lambda^{-1}X_{k-1,k-1}\right).
\end{equation}
By the change of variables
\[X_{i,j}\mapsto \lambda X_{i,j},\]
we get
\begin{equation*}
	\Theta_{k,N,\psi}\left(\lambda I\right)= \frac{1}{q^{\binom{k}{2}n^{2}}}\sum_{U\in N}\Theta_{\theta}\left(\lambda h_{I;X,Y,Z}\right)\overline{\psi}\left(X_{1,1},\ldots,X_{k-1,k-1}\right).
\end{equation*}
By Theorem \ref{greenthm}, we have $\Theta_{\theta}\left(\lambda\cdot  h_{I;U}\right)=\theta(\lambda)\Theta_{\theta}\left(h_{I;U}\right)$, and so
\begin{equation} \label{char_on_center}
\Theta_{k,N,\psi}\left(\lambda I\right)=\theta(\lambda)\Theta_{k,N,\psi}\left(I\right)=\theta(\lambda)\mathrm{dim}\left(\pi_{k,N,\psi}\right).
\end{equation}
By Theorem \ref{int_dim}, this proves the case  $\lambda\in\bff^*$.

If $\lambda \in \bff^*_{d} \subseteq \bff^*_{n}$ is an eigenvalue of $s$ and  $1<d\mid n$ is such that $\bff_d$ is generated by $\lambda$ over $\bff$, we have by \eqref{charformula}
\begin{equation} \label{char_semi_form}
\Theta_{k,N,\psi}\left(s\right)= \frac{1}{q^{\binom{k}{2}n^{2}}}\sum_{U\in N}\Theta_{\theta} \left(h_{s;U}\right) \overline{\psi}\left(s^{-1}X_{1,1},\ldots,s^{-1}X_{k-1,k-1}\right).
\end{equation}
In order to compute $\Theta_{\theta}(h_{s;U})$, we need to find conditions for $X_{i,j}$, such that $h_{s;U}$ will have a fixed number of Jordan blocks. This is equivalent to saying that $h_{s;U}-\lambda I_{kn}$ will have a given kernel dimension, or a given rank.
Rank and trace are invariant under conjugation, so let us denote by $h_{s_{\rho},U_{\rho}}$, the matrix $h_{s;U}$ conjugated by $\Delta^k(R)\Delta^k\left( \Delta^{d^\prime}\left(T\right)\right)$, where $R$ and $T$ are defined by $s$ in \eqref{defrts}:
\begin{equation*}
h_{s_\rho;U_\rho} := \Delta^k\left( \Delta^{d^\prime}\left(T^{-1}\right)\right)\Delta^k\left( R^{-1}\right) h_{s;U} \Delta^k(R)\Delta^k\left( \Delta^{d^\prime}\left(T\right)\right).
\end{equation*}
We have a matrix in $\gl[\bff_d]{kn}$ and our goal is to find out how many matrices of the form
\begin{equation*}
h_{s_\rho;U_\rho} -\lambda I_{kn}= h_{s_{\rho}-\lambda I_n;U_\rho},
\end{equation*}
where $U$ varies, have a given rank $\ell$.

First, notice that by the invariance of rank under elementary row and column operations on $h_{s_{\rho}-\lambda I_n;U_\rho}$, we can use the nonzero elements on the diagonal of $s_\rho-\lambda I_n$ to cancel the corresponding elements of $(X_{a,b})_\rho$.  These elementary operations map the sequence of matrices $\{(X_{a,b})_\rho \}_{1 \le a \le b \le k-1}$ $\bff_d$-linearly to the sequence
\begin{equation}
\Big\{ (\widehat{X}_{a,b})_\rho = \begin{pmatrix}
x^{(a,b)}_{0,0;0} & \cdots & x^{(a,b)}_{{d^\prime-1},0;0} \\
\vdots & \ddots & \vdots  \\
x^{(a,b)}_{0,0;d^\prime-1} & \cdots & x^{(a,b)}_{{d^\prime-1},0;d^\prime-1}
\end{pmatrix} \in M_{d^\prime}\left(\bff_d\right) \Big\}_{1 \le a \le b \le k-1}.
\end{equation}
The dimension of the kernel of this map is $\binom{k}{2}(n-d')d'$, corresponding to the number of elements we canceled. Hence, the number of matrices $h_{s_{\rho}-\lambda I_n;U_\rho}$ of rank $\ell$ is $(q^{d})^{\binom{k}{2}(n-{d^\prime}){d^\prime}}$ times the number of matrices of the form
\begin{equation} \label{asemicanceled}
    A:=\begin{pmatrix}
	    (\widehat{X}_{1,1})_\rho  & \cdots&(\widehat{X}_{1,k-2})_{\rho}&(\widehat{X}_{1,k-1})_{\rho} \\
	    0  & \cdots &(\widehat{X}_{2,k-2})_{\rho}&(\widehat{X}_{2,k-1})_{\rho} \\
		\vdots&&\vdots&\vdots\\
		0 & \cdots & 0 &(\widehat{X}_{k-1,k-1})_{\rho}
	\end{pmatrix}\in M_{(k-1)d^\prime}(\bff_d)
\end{equation}
of rank $\ell-k(n-d^\prime)$.
According to the character formula \eqref{green}, we can calculate $\Theta_\theta(h_{s;\unimat})$. In this case $m=kn$, $g=h_{s;\unimat}$  and
$$t=\mathrm{dim}\ \mathrm{ker}(h_{s;\unimat}-I)=kn-\mathrm{rk}(h_{s;\unimat}-I)=kn-k(n-d^\prime) -\mathrm{rk}A=kd^\prime -\mathrm{rk}A.$$
Thus
\begin{align} \label{char_semi_1}
\Theta_\theta \left(h_{s;\unimat}\right) ={}& (-1)^{kn-1}\left[ \sum \limits _{i=0}^{d-1} \theta(\lambda^{q^i})\right] (1-q^d)(1-(q^d)^2)\cdots(1-(q^d)^{kd^\prime  -\mathrm{rk}A-1})\nonumber\\
={}&(-1)^{kn-1}\left[ \sum \limits _{i=0}^{d-1} \theta(\lambda^{q^i})\right] (q^d;q^d)_{kd^\prime -\mathrm{rk}A-1}.
\end{align}
Now, by \eqref{char_semi_1} and Lemma \ref{lem_tr_conj}, \eqref{char_semi_form} can be written as
\begin{equation}\label{semisum}
\begin{split}
\Theta_{k,N,\psi}\left(s\right)= \frac{(-1)^{kn-1}(q^{d})^{\binom{k}{2}(n-{d^\prime}){d^\prime}}}{q^{\binom{k}{2}n^{2}}}&\left[ \sum \limits _{i=0}^{d-1} \theta(\lambda^{q^i})\right]\sum_{A}(q^d;q^d)_{kd^\prime -\mathrm{rk}A-1}\\&\cdot\prod_{i=1}^{k-1} \overline \psi_0 \left(\mathrm{Tr}_{\bff_{d}/\bff}\left(\lambda ^{-1} \cdot \sum \limits _{m=0} ^{d^\prime-1} x^{(i,i)}_{m,0;m}\right)\right),
\end{split}
\end{equation}
where the sum is over matrices $A$ as in \eqref{asemicanceled}. By the character formula \eqref{green}, the RHS of \eqref{semisum} is $(-1)^{k(n-d^\prime)}\left[ \sum \limits _{i=0}^{d-1} \theta(\lambda^{q^i})\right]$ times the RHS of \eqref{splitsumdim}, when one replaces $n$ with ${d^\prime}$, $q$ with $q^d$ and $\psi_0$ with
\begin{equation}
\psi^{\prime}_0:\bff_d \to \mathbb{C}^{*}, \quad \psi^{\prime}_0(x)=\psi_0 \Big( \mathrm{Tr}_{\bff_{d}/\bff}(\lambda ^{-1}x) \Big).
\end{equation}
Thus, the RHS of \eqref{semisum} is equal to $\mathrm{dim}\left(\pi_{k,N,\psi}\right)$ (which was calculated in Theorem \ref{st_thm}) after the substitution of $n,q,\psi_0$ with the relevant values. Hence,
\begin{equation}
\begin{split}
\Theta_{k,N,\psi}\left(s\right)=(-1)^{k(n-d^\prime)}\left[ \sum \limits _{i=0}^{d-1} \theta(\lambda^{q^i})\right](q^d)^{(k-2) \frac{d^\prime(d^\prime-1)}{2}}\frac{\left|\gl[\bff_d]{d^\prime}\right|}{q^n-1},
\end{split}
\end{equation}
as desired.
\qed 

\section{Proof of Theorem \ref{mainthrm}} \label{mainth_sec}

	Notice first that by part (III) of Lemma \ref{lemand}, the coefficients in both \eqref{mainth1} and \eqref{mainth2} are positive integers, unless $k=2$ in which case they may also be zero.
	
	Representations of a finite group are equivalent if the corresponding characters coincide. Hence, both parts of the theorem are equivalent to
	\begin{equation}\label{deq}
	\forall g \in \gl{n}: \, \Theta_{k;N,\psi}(g) = \sum_{\ell \mid n} a_{k;n,\ell}(q) \cdot \Theta_{\mathrm{Ind}_\ell}(g).
	\end{equation}
	We prove now \eqref{deq} for any $g\in\gl{n}$.
	If $g$ is not semisimple or does not come from $\bff_n$ then the LHS of \eqref{deq} is zero by parts (I) and (II) of Theorem \ref{int_char_thm}. The RHS of \eqref{deq} is also zero on such elements by Lemma \ref{lemindch}.
	
	Let $g$ be a semisimple element, which comes from $\bff_d\subseteq \bff_n$ and $d\mid n$ is minimal. Let $\lambda$ be an eigenvalue of $s$, which generates $\bff_d$ over $\bff$. For such $g$, part (III) of Theorem \ref{int_char_thm} and Lemma \ref{lemindch} imply that \eqref{deq} is equivalent to
	\begin{equation} \label{deq1}
	(-1)^{k(n-{d^\prime})} \left[ \sum \limits _{i=0}^{d-1} \theta(\lambda^{q^i})\right] q^{(k-2)\frac{n({d^\prime}-1)}{2}}\cdot  \frac{\left|\mathrm{GL}_{{d^\prime}}(\bff_{d})\right|}{q^n-1}=
	\sum_{\ell:\ d\mid \ell \mid n} a_{k;n,\ell}(q) \frac{\left|\gl[\bff_d]{d^\prime}\right|}{q^\ell-1}\left[\sum\limits_{i=0}^{d-1}\theta(\lambda^{q^i})\right],
	\end{equation}
	where $d^\prime=n/d$. Proving the following identity will establish \eqref{deq1}:
	\begin{equation} \label{deq2}
	\frac{(-1)^{k(n-{d^\prime})} q^{(k-2)\frac{n({d^\prime}-1)}{2}}}{q^n-1}=
	\sum_{\ell:\ d\mid \ell \mid n} \frac{a_{k;n,\ell}(q)}{q^\ell-1}.
	\end{equation}
	Using \eqref{anlq}, the RHS of \eqref{deq2} is
	\begin{equation} \label{reqiden}
	 \sum_{\ell:\ d\mid \ell \mid n} \sum_{m: \, \ell \mid m \mid n} \frac{\mu\left(\frac{m}{\ell}\right) (-1)^{k(n- \frac{n}{m})} q^{(k-2)\frac{n}{2} (\frac{n}{m}-1)}}{q^n-1}.
	\end{equation}
	We simplify \eqref{reqiden} using \eqref{mobius_prop}:
	\begin{equation}
	\begin{split}
	 \sum_{\ell:\ d\mid \ell \mid n} \sum_{m: \, \ell \mid m \mid n} &\frac{\mu\left(\frac{m}{\ell}\right) (-1)^{k(n- \frac{n}{m})} q^{(k-2)\frac{n}{2} (\frac{n}{m}-1)}}{q^n-1}\\
	 &= \sum_{m: \, d \mid m \mid n} \frac{(-1)^{k(n-\frac{n}{m})} q^{(k-2)\frac{n}{2} (\frac{n}{m}-1)}}{q^n-1} \sum_{\ell:\ d\mid \ell \mid m} \mu\left(\frac{m}{\ell}\right)\\
	&=\sum_{m: \, d \mid m \mid n} (-1)^{k(n-\frac{n}{m})} \frac{ q^{(k-2)\frac{n}{2} (\frac{n}{m}-1)}}{q^n-1}\delta_{d,m}\\
	&= (-1)^{k(n-\frac{n}{d})} \frac{ q^{(k-2)\frac{n}{2} (\frac{n}{d}-1)}}{q^n-1},
	\end{split}
	\end{equation}	
	which is the LHS of \eqref{deq2}. Hence the proof is completed.
\section{Proof of Theorem \ref{st_thm}} \label{st_sec}

	Representations of a finite group are equivalent if the corresponding characters coincide. Hence, the theorem is equivalent to
	\begin{equation}\label{st_char_eq}
	\forall g \in \gl{n}: \,\quad \Theta_{k,N,\psi}(g) = \Theta_{\theta\upharpoonright_{\bff^*_n}}(g) \cdot \mathrm{St}^{k-1}(g),
	\end{equation}
	where we use the notation $\mathrm{St}$ also for the character of the Steinberg representation.
	We prove now \eqref{st_char_eq} for any $g\in\gl{n}$.
	
	We first prove \eqref{st_char_eq} for $k=1$. Note that $N=\left\{I_n\right\}$ and so
	\begin{equation}
	V_{\pi_{1,N,\psi}}=\left\{ v\in V_{\pi_\theta}\ \left|\ \pi(I_n)v=v\right.\right\}=V_{\pi_\theta} .
	\end{equation}
	Hence $\pi_{1,N,\psi}(g) = \pi_{\theta}(g)$ as needed.
	
	Now assume $k \ge 2$. If the semisimple part $s$ of $g$ does not come from $\bff_n$, or $g$ is not semisimple, then $\Theta_{k,N,\psi}(g)=0$ by Theorem \ref{int_char_thm}. From Theorem \ref{greenthm}, we have $\Theta_{\theta\upharpoonright_{\bff^*_n}}(g) = 0$. Hence, \eqref{st_char_eq} is proved in that case.
	
	Otherwise, $g=s$ is a semisimple element which comes from $\bff_d \subseteq \bff_n$ and $d \mid n$ is minimal. We begin by calculating the character value $\mathrm{St}(g)$. For any prime $p$, let $m_p$ be the $p$-part of $m$. By \cite[Thm.~6.5.9]{carter1993},
	\begin{equation}
	\mathrm{St}(g)=\varepsilon_{\gl[\overline{\bff}]{n}} \varepsilon_{C(g)^{\circ}} \left| C(g)^{\bff} \right|_{\mathrm{char}(\bff)},
	\end{equation}
	where $\varepsilon_G$ is $(-1)$ to the power of the $\bff$-rank of $G$, $C(g)$ is the centralizer of $g$ in $\gl[\overline{\bff}]{n}$, $C(g)^{\circ}$ is its identity component and $C(g)^{\bff}$ is the subgroup of $\bff$-rational points in $C(g)$. The $\bff$-rank of $\gl[\overline{\bff}]{n}$ is $n$. Let $\rho = \left(1^{n/d}\right)$, a partition of $d'=\frac{n}{d}$ and let $f$ be the characteristic polynomial of $s$.
	 By \S \ref{anal_jord_sec}, the centralizer $C(g)^{\bff}$ is isomorphic to $C(L_{f,\rho})^{\bff}$, which in turn is isomorphic to $\gl[\bff_{d}]{d'}$ (cf. \cite[Lem.~2.4]{green1955} and the discussion preceding it). Thus,  $\varepsilon_{C(g)^{\circ}} =\varepsilon_{\gl[\overline{\bff}]{d'}}= (-1)^{d'}$ and
	\begin{equation}
	\left| C(g)^{\bff} \right| =q^{\sum_{i=1}^{d'} d(d'-i)} \prod_{k=1}^{d'}\left(q^{d k}-1\right), \quad \left| C(g)^{\bff} \right|_{\mathrm{char}(\bff)} = q^{\frac{n(d'-1)}{2}}.
	\end{equation}
	The discussion shows that	
	\begin{equation}\label{stgval}
	\mathrm{St}(g)=(-1)^{n-d'}q^{ \frac{n(d'-1)}{2}}.
	\end{equation}	
	By Theorem \ref{greenthm},
	\begin{equation}\label{t2re}
		\begin{split}
		\Theta_{\theta\upharpoonright_{\bff^*_n}}(g)&=(-1)^{n-1}\left[\sum\limits_{\alpha=0}^{d-1}\theta(\lambda^{q^\alpha})\right](1-q^d)(1-({q^d})^2)\cdots(1-({q^d})^{d'-1})\\
		&= (-1)^{n-d'} \left[\sum\limits_{\alpha=0}^{d-1}\theta(\lambda^{q^\alpha})\right](q^d-1)(q^{2d}-1)\cdots(q^{n-d} -1) \frac{q^n-1}{q^n-1}\\
		&= (-1)^{n-d'} \left[\sum\limits_{\alpha=0}^{d-1}\theta(\lambda^{q^\alpha}) \right] \frac{\left|\mathrm{GL}_{{d^\prime}}(\bff_{d}) \right|}{(q^n-1) q^{\frac{n(d'-1)}{2}}},
		\end{split}
			\end{equation}
	where $\lambda$ is an eigenvalue of $g$. By Theorem \ref{int_char_thm}
	\begin{equation}\label{t2p}
	\Theta_{k,N,\psi}(g) = (-1)^{k(n-{d^\prime})}  q^{(k-2)\frac{n({d^\prime}-1)}{2}} \cdot \left[ \sum \limits _{i=0}^{d-1} \theta(\lambda^{q^i})\right] \cdot  \frac{\left|\mathrm{GL}_{{d^\prime}}(\bff_{d})\right|}{q^n-1}.
	\end{equation}
	Multiplying \eqref{t2re} by \eqref{stgval} raised to the $(k-1)$-th power, we get \eqref{t2p} as needed.	

\section*{Acknowledgments}
We are grateful to the second author's advisor, David Soudry, for suggesting the problem and for many helpful discussions during our work on the case $k=3$.
	
	We are thankful to Dipendra Prasad for interesting discussions. We are indebted to Dror Speiser for useful conversations, and in particular for suggesting the link with the Steinberg representation.

\bibliographystyle{alpha}
\bibliography{references}

\end{document}